\documentclass[sts]{imsart}

\RequirePackage{amsthm,amsmath,amsfonts,amssymb}
\RequirePackage[authoryear]{natbib}
\RequirePackage[colorlinks,citecolor=blue,urlcolor=blue]{hyperref}

\usepackage{url}
\usepackage{times}
\usepackage{bm}
\usepackage{color}
\usepackage[plain,noend]{algorithm2e}
\usepackage{tikz}
\usepackage{tikz-3dplot}
\usepackage{comment}
\usepackage{ulem}

\RequirePackage{graphicx}

\startlocaldefs

\newtheorem{theorem}{Theorem}[section]
\newtheorem{lemma}[theorem]{Lemma}
\newtheorem{corollary}{Corollary}[section]
\newtheorem{proposition}{Proposition}[section]
\theoremstyle{remark}
\newtheorem{definition}[theorem]{Definition}
\newtheorem{example}{Example}[section]

\newtheorem{remark}{Remark}[section]

\endlocaldefs

\begin{document}

\begin{frontmatter}
\title{Geometric Conditions for the Discrepant Posterior Phenomenon 
and Connections to Simpson's Paradox}
\runtitle{Discrepant Posterior Phenomenon}

\begin{aug}
\author[A]{\fnms{Yang} \snm{Chen}
\ead[label=e1]{ychenang@umich.edu}},
\author[B]{\fnms{Ruobin} \snm{Gong}\ead[label=e2]{ruobin.gong@rutgers.edu}}
\and
\author[C]{\fnms{Min-ge} \snm{Xie}\ead[label=e3]{mxie@stat.rutgers.edu}}


\address[A]{Yang Chen is Assistant Professor, Department of Statistics and the Michigan Institute for Data Sciences (MIDAS), University of Michigan, Ann Arbor, USA \printead{e1}.}

\address[B]{Ruobin Gong is Assistant Professor, Department of Statistics, Rutgers University – New Brunswick, USA \printead{e2}.}

\address[C]{Min-ge Xie is Distinguished Professor,
Department of Statistics, Rutgers University – New Brunswick, USA \printead{e3}. The research is supported in part by NSF research grants: DMS1811083, 
 DMS1812048, DMS2015373 and DMS2027855.}

\end{aug}

\begin{abstract}
The {\it discrepant posterior phenomenon} (DPP) is a counter-intuitive phenomenon that can frequently occur in a Bayesian analysis of multivariate parameters. It refers to the phenomenon that 
a parameter estimate based on a 
posterior is more extreme than both of those inferred based on either the prior or the likelihood alone. Inferential claims that exhibit DPP defy the common intuition that the posterior is a prior-data compromise, and the phenomenon can be surprisingly ubiquitous in well-behaved Bayesian models. In this paper we revisit this phenomenon and, using point estimation as an example, derive conditions under which the DPP occurs in Bayesian models with exponential quadratic likelihoods
and conjugate multivariate Gaussian priors. The family of exponential quadratic likelihood models includes Gaussian models and those models with local asymptotic normality property. 
We provide an intuitive geometric interpretation of the phenomenon and show that there exists a nontrivial space of marginal directions such that the DPP occurs. We further relate the phenomenon to the Simpson's paradox and discover their deep-rooted connection that is associated with marginalization. 
We also draw connections with Bayesian computational algorithms when difficult geometry exists. Our discovery demonstrates that DPP is more prevalent than previously understood and anticipated.  Theoretical results are complemented by numerical illustrations. Scenarios covered in this study have implications for parameterization, sensitivity analysis, and prior choice for Bayesian modeling. 
\end{abstract}


\end{frontmatter}

\section{Introduction}

In Bayesian analysis, the posterior distribution provides a probabilistic summary that incorporates both the prior knowledge and what can be learned from data. 
Bayesian inferential statements on model parameters are derived solely from the posterior distribution. In many applications for which the model parameter is multi-dimensional, we are only interested in inference about a certain marginal parameter, say $\eta$, where $\boldsymbol{\theta} = (\eta^\top, \alpha^\top)^\top $ is the full model parameter of dimension $d > 1$, and $\alpha$ is the nuisance parameter. For such inference problems, 
the Bayesian approach typically assigns a prior to the full parameter $\boldsymbol{\theta}$, and an estimate of the target parameter $\eta$ obtained from the marginal posterior of $\eta$ (cf., e.g., \citealt{efron1986isn,wasserman2007isn}). This Bayesian inference approach is coherent and supported by probability theory, and it has been extensively used in practice; see discussions on multi-parameter models in~\citet{gelman2013bayesian} and references therein. However, if inference is about the marginal parameter $\eta$, we show in this note that the so-called {\it discrepant posterior phenomenon} (DPP) may occur frequently. 


Bayesian posterior inference is often viewed as a combination of information from the prior and likelihood. Thus, we generally expect it to be a compromise between the two. Estimates based on the posterior are expected to be more moderate than either of the corresponding estimates from the prior or the likelihood. For example, in a Gaussian conjugate model with {an unknown mean parameter of interest} and known variance, the posterior mean is a weighted average of the prior mean and the maximum likelihood estimate (MLE). Therefore, the posterior estimate lies between the estimates based on the prior and the likelihood. What is lesser known is that, when we have multiple model parameters (parameter of interest plus nuisance parameters) and the prior is informative, the practice of marginalizing a full Bayesian posterior to the parameter of interest can lead to counter-intuitive posterior inference. The DPP occurs when an estimate derived from the (marginal) posterior takes a value that is more extreme than those based on either the prior or the data. That is, 
the estimate of $\eta$ derived from the posterior is more extreme than both of those inferred based on either the prior or the likelihood alone. 
The phenomenon is counterintuitive, for it defeats the general expectation of the posterior as a prior-data compromise for the target parameter. We investigate this phenomenon, and reveal that DPP bears a structural resemblance to the Simpson's paradox.  

\subsection{Literature on DPP}
The DPP was first reported in \citet{xie2013incorporating} in a study of a Binomial clinical trial conducted by Johnson \& Johnson Inc. Both expert opinions (forming an informative prior) and data from the clinical trial agreed that the improvement $\eta = p_1 - p_0$, from the control success rate ($p_0$) to the treatment success rate ($p_1$), is around $10\%$. However, the marginal posteriors of $\eta$ from several candidate full Bayesian models on $(p_0, p_1)$ suggested that the improvement is over $20\%$ \cite[Table 3]{xie2013incorporating}.
Figure 2 therein shows that the marginal posterior of $\eta$ peaks outside of the marginal prior distribution of $\eta$ and the profile likelihood function of $\int L(p_0, p_0 + \eta | {\rm data}) d p_0$. Here, $L(p_0, p_1 | {\rm data})$
is the joint likelihood of $(p_0, p_1)$, and the marginal prior distribution of $\eta$ and the profile likelihood function are more or less in agreement. 
\cite{xie2013incorporating} has also explored  different prior models, including  independent, dependent and hierarchical bivariate Beta priors, for the full parameter $\boldsymbol{\theta} = (p_0, p_1)^\top$, but similar DPP occurs.



The DPP is not a mathematical oversight. Provided that both the prior and likelihood specifications are correct and the prior is proper, the Bayes calculation esures that conclusions based on the marginal posterior of $\eta$ are necessarily correct, whether or not DPP is present. Practically, however, DPP can lead to undesirable complications. For instance, in the example from \citet{xie2013incorporating}, should we trust the conclusion that the improvement $p_1 - p_0$ is over $20\%$ based on the marginal posterior of $\eta$? Many may choose to question the prior 
or the data model specifications. However, as we will see in later sections, the phenomenon is surprisingly ubiquitous in  well-behaved Bayesian models. 
Subsequent discussions in~\citet[Section 6.2]{xie2013confidence}, \citet{robert2013} and \citet{xie2013confidence_rejoinder} suggest that the DPP is commonplace in multivariate Bayesian analysis. 
The observation generated further discussions on whether it is necessary to require some alignment of the prior given the likelihood, which in turn raised questions and disagreement about whether data-dependent priors should be used. 


\subsection{Our Key Contributions}
In this paper, we revisit the DPP phenomenon and provide under a linear parameter setting a set of explicit conditions under which DPP occurs. We also present a clear connection for this Bayesian phenomenon to Simpson's paradox that is often discussed under frequentist contexts.  
Specifically, to be
precise 
mathematically 
but without loss of generality, 
we 
study DPP using point estimation and consider the  parameter of interest a linear combination of the full parameter $\boldsymbol{\theta}$, $\eta = \boldsymbol{\lambda}^{\top} \boldsymbol{\theta}$, for a given $\boldsymbol{\lambda} \in \mathbb{R}^d$, $d>1$. 
In particular, technical examinations of the DPP is performed for 
a class 
of models,
in which observation ${\bf y}$ are assumed to exhibit an exponential quadratic likelihood. That is, $L(\boldsymbol{\theta}; {\bf y}) \propto \exp\left\{- q(\boldsymbol{\theta})\right\}$, where $q(\boldsymbol{\theta})$ is a quadratic function of $\boldsymbol{\theta}$.
The family of the exponential quadratic likelihood includes Gaussian models and also those with local asymptotic normality (LAN) property as special cases. For the ease of presentation in this study, 
the prior distribution is assumed to be fully specified as a multivariate Gaussian distribution, which is a conjugate for the Gaussian and LAN likelihood. 
We demonstrate that DPP is more prevalent than previously anticipated, supported by theoretical results on the probability that DPP occurs and also accompanied by its numerical estimation using simulation experiments. Furthermore, we show that heterogeneous variances on the independent component dimensions of the parameter can result in DPP, demonstrating that DPP is not a result of correlations among parameters due to ``complex'' geometry. 
Moreover, the connection to Simpson's paradox is first discovered in this article. The discovery revealed that both mind-boggling phenomena share the same underlying mathematical structure, although the entities involved in the DPP phenomenon are point estimators based on the prior distribution and the likelihood function, while the typical case of Simpson's paradox are on different groups of data.

We conducted extensive numerical study under both the Gaussian conjugate model and the Binomial model with various prior choices, where the focus is on marginal parameters that are linear contrasts of the full model parameter. Under the Gaussian conjugate model setting, we find that (I) 
DPP can occur for a reasonably large sample size,
but the probability that DPP occurs decreases as the sample size increases when the prior fixed; (II) DPP can occur in both settings of correlated and uncorrelated parameters, but models with correlated parameters are more prone to DPP, compared to their counterparts with uncorrelated parameters;
and (III) misalignment between the prior mean and the MLE exacerbates DPP, even more so in models with correlated parameters.  Under the Binomial model setting when the quantity of interest is the linear contrast of two probabilities, we find that (I) a positive prior correlation between the two proportion parameters exacerbates the DPP, whereas a negative prior correlation alleviates it; (II) An increasingly larger prior variance, i.e. flatter relative to the likelihood, tend to alleviate the DPP; and (III) the DPP could be avoided if the prior variances were specified with unknown hyperparameters to accord with the data. Finally, although the DPP in nonlinear non-Gaussian models can behave differently from the exponential quadratic models, the key points discovered still has implications for the general setting.


These results provide a fuller picture and further understanding of the DPP, including a connection to Simpson's paradox. 
The development can provide intuitions on how to mitigate and interpret the DPP in the presented cases. It can serve as precautions for practitioners of Bayesian inference when highly informative priors are desired (e.g. in astrophysics applications~\citep{chen2019calibration}, in meta analysis in medical sciences~\citep{rhodes2016implementing}, in computational linguistics~\citep{lapata2004verb}, and in cognitive modeling~\citep{lee2018determining}); and more importantly provide practical guidance towards prior specification (including dispersed/weakly informative priors, hierarchical priors, and re-parameterization for prior-likelihood curvature alignment, see Sections 2 and 5 of \citet{gelman2013bayesian} for more detailed discussions and examples) for Bayesian inference, to mitigate and possibly avoid the DPP. This recommendation aligns with the idea promoted in \citet{gelman2017prior}, ``a prior can in general only be interpreted in the context of the likelihood with which it will be paired''. Furthermore, we believe that models that demonstrate the DPP shall be adopted as test cases for works that seek to form objective~\citep{berger2015overall} or non-informative~\citep{yang1996catalog} priors, works on prior sensitivity checks~\citep{berger1990robust}, and works on quantifying prior influences~\citep{reimherr2014being,jones2020quantifying}. And this is due to the fact that the DPP defeats intuitions on prior-likelihood interactions thus corresponding models can reveal properties of various prior specification/quantification methods.

\subsection{Outline}

The remainder of the paper is organized as follows.  Section~\ref{sec:DPP_existence} provides a precise definition of DPP considered in this article, with brief discussions on the prevalence of DPP.  Notably, with the exception of certain special cases, there always exist certain marginal directions along which DPP occurs. Section~\ref{section:conditions_DPP} derives specific conditions under which DPP occurs (and does not occur) in a model from the family of exponential quadratic likelihood, accompanied by numerical examples to illustrate the prevalence of DPP and the theoretical conditions.  Section~\ref{section:geometry} offers geometric and intuitive interpretations of the conditions for which the DPP occurs or not, and establishes a connection between DPP and the Simpson's paradox in the model space. We revisit the Binomial example given in~\citet{xie2013incorporating} in Section~\ref{section:binomialmodel}. In the highly nonlinear model, the DPP phenomenon is more complicated than the Gaussian case. Instead of giving analytical solutions for DPP, we examine from both theoretical and numerical perspectives the DPP for the Binomial model. Finally, we conclude the paper in Section~\ref{section:summary} with discussions on the DPP and its implications for parametrization, sensitivity analysis, and choice of priors in Bayesian analysis.


\section{Definition and Existence of DPP}
\label{sec:DPP_existence} 

\subsection{A Point-wise Definition of DPP}
The DPP occurs whenever an estimate from the marginal posterior does not lie between the corresponding estimates from the prior and the likelihood. In Figure~\ref{fig:dpp_beta_binomial_illustration}, we illustrate the DPP in a simple two-dimensional Beta-Binomial conjugate model, where the quantity of interest is a one-dimensional contrast, the same as the example studied in \citet{xie2013incorporating}. In this numerical experiment, we assume that the model is $x_j\sim {\rm Binomial}(n_j, p_j)$ independently for $j=0, 1$; and the priors are $p_j\sim {\rm Beta}(a_j, b_j)$ independently for $j=0, 1$; where $x_0=30, n_0=70$, $x_1=30, n_1=60$, $a_0=15, b_0=4$, $a_1=45, b_1=7$. The quantity of interest is $\eta = p_1-p_0$. As can be seen from the marginal densities in Figure~\ref{fig:dpp_beta_binomial_illustration}, the information for $\eta = p_1-p_0$ is consistent in the prior and the profile likelihood functions since the curves obtained by projections to the direction of  $\eta = p_1-p_0$ are almost identical, but the marginal posterior distribution is quite different.  
If we consider point estimators,
the marginal posterior mode of $\eta$ is {$\eta^p = 0.17$}, which is located to the right of that of both its prior and the likelihood {(i.e., $\eta^\pi = 0.06$ and $\eta^L = 0.07$)}, respectively.

To understand the essence underlying DPP and also simplify our presentations in 
explicit and precise mathematical forms, we restrict our attention to a point-wise definition of DPP. 
Specifically, we assume our parameter of interest is $\eta = \boldsymbol{\lambda}^{\top} \boldsymbol{\theta}$, where $\boldsymbol{\theta} \in \Theta$ is the full model parameter of dimension $d > 1$, and $\Theta$ is a non-degenerate continuous subspace of $\mathbb{R}^d$. Denote the MLE of $\boldsymbol{\theta}$ by $\boldsymbol{\mu}^L$, the prior mean of $\boldsymbol{\theta}$ by $\boldsymbol{\mu}^{\pi}$, and the posterior estimate of $\boldsymbol{\theta}$ by $\boldsymbol{\mu}^{p}$. Denote $\eta^p = \boldsymbol{\lambda}^{\top}\boldsymbol{\mu}^{p}$ as the posterior point estimate of $\eta$ (mean, median, or mode), $\eta^{\pi}$ the prior mean of $\eta$, and let $\eta^L$ be an estimate of $\eta$ derived from the likelihood function, such as the MLE. A formal point-wise definition based on MLE and prior and posterior means is given below. 

\begin{definition}[Point-wise DPP on posterior mean] \label{def:dpp}
We say that the {\it discrepant posterior phenomenon} (DPP) occurs, if 
\begin{equation}
    \left(\eta^p - {\eta}^{\pi}\right)\left(\eta^p - {\eta}^{L}\right)> 0. \label{eqn:DPP_eqn}
\end{equation}
\end{definition}

We consider $\eta^p$ to be either the posterior mean or posterior mode in this paper. One may wish to also define the discrepant posterior phenomenon for more general types of estimators, such as point estimators other than expectations and the MLE, as well as interval estimators. To maintain clarity of the current paper, we defer discussions about alternative definitions of DPP to future work, noting here that such definitions are conceivable. See also Section~\ref{section:summary} for further discussions.

\begin{figure}[ht]
    \centering
    \includegraphics[width=0.9\textwidth, height=70mm]{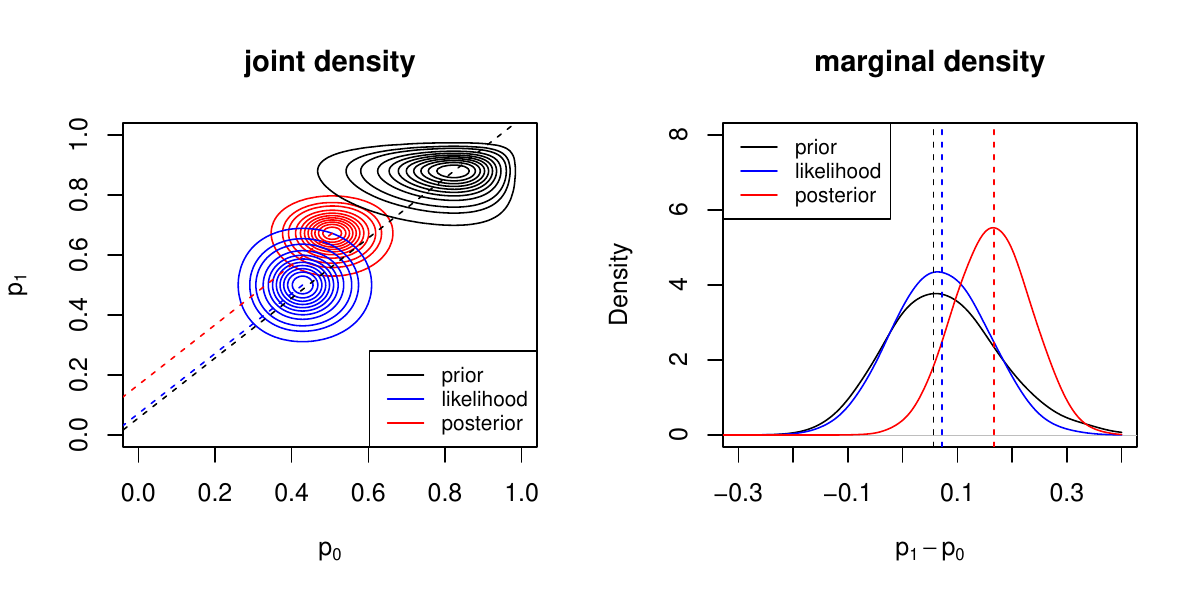}
   \caption{A demonstration of DPP with a two dimensional Beta-Binomial model where the quantity of interest is $\eta = p_1-p_0$, the priors for $p_0$ and $p_1$ are independent Beta distributions with degrees of freedom $(15, 4)$ and $(45, 7)$ respectively, and the data consist of $30$ out of $70$ 'successes' and $30$ out of $60$ 'successes' in independent Bernoulli experiments with 'success probabilities' $p_0$ and $p_1$ respectively. On the left panel, we show contour plots of the joint prior $\pi(p_0, p_1)$, likelihood function $L(p_0, p_1|{\rm data})$, and posterior function $p(p_0, p_1|{\rm data})$. On the right panel, we show projections (marginals) of $\pi(p_0, p_1)$,$L(p_0, p_1|{\rm data})$, and $p(p_0, p_1|{\rm data})$ onto the direction of $\eta = p_1 - p_0$. In both panels, the black, blue, and red contours/curves correspond to the prior, likelihood, and posterior densities respectively.
    }
    \label{fig:dpp_beta_binomial_illustration}
\end{figure}

\subsection{Prevalence of DPP: the First Look}
Let us briefly discuss the existence of DPP from the definition; later sections expand on the geometric illustrations of DPP. 
Denote the $d\times 3$ matrix $\boldsymbol{M} = (\boldsymbol{\mu}^{p}, \boldsymbol{\mu}^{\pi},\boldsymbol{\mu}^L)$, which collects the posterior, prior and likelihood point estimates, and the $1\times 3$ vector $\boldsymbol{\zeta}^\top = \boldsymbol{\lambda}^\top \boldsymbol{M} = (
\eta^p, \eta^\pi, \eta^L) $. We will show with a simple argument of analytic geometry and linear algebra that, for any given combination of linearly independent point estimates of the posterior, prior and likelihood, i.e. $(\boldsymbol{\mu}^p,\boldsymbol{\mu}^\pi,\boldsymbol{\mu}^L)$, there exists a nontrivial (i.e. non-degenerate with a positive volume) subspace of $\mathbb{R}^d$ such that for any $\boldsymbol{\lambda}$ that takes values in this subspace, the DPP appears for $\eta = \boldsymbol{\lambda}^\top\boldsymbol{\mu}$. 

{When $d \ge 3$, for a given $\boldsymbol{M}$, any value $\boldsymbol{\lambda}$ that makes $\boldsymbol{\zeta}^\top=\boldsymbol{\lambda}^\top\boldsymbol{M}$ satisfy the inequality~\eqref{eqn:DPP_eqn} results in DPP. If $d > 3$, the values of $\boldsymbol{\zeta}$ that satisfy the inequality~\eqref{eqn:DPP_eqn} span a nontrivial subspace of $\mathbb{R}^3$.
If $\boldsymbol{\mu}^{\pi}, \boldsymbol{\mu}^L$ and $\boldsymbol{\mu}^p$ are not collinear,
the equation $\boldsymbol{\zeta} = \boldsymbol{M}^{\top}\boldsymbol{\lambda}$ asserts three linearly independent constraints on  $\mathbb{R}^d$. Then for each $\boldsymbol{\zeta}$ that satisfies~\eqref{eqn:DPP_eqn}, there exists a subspace of dimension $d-3$ in which each $\boldsymbol{\lambda}$ value is a solution to the equation  $\boldsymbol{\zeta} = \boldsymbol{M}^{\top}\boldsymbol{\lambda}$. On the other hand, if $d=3$ and $M$ is of full rank,  for each $\boldsymbol{\zeta}$ that satisfies~\eqref{eqn:DPP_eqn}, the equation $\boldsymbol{\zeta}=\boldsymbol{M}^\top \boldsymbol{\lambda}$ has a unique solution of $\boldsymbol{\lambda}\in\mathbb{R}^3$. 
Thus, for any linearly independent combination of $\boldsymbol{\mu}^p,\boldsymbol{\mu}^\pi,\boldsymbol{\mu}^L$, DPP would occur for $\boldsymbol{\lambda}$ values in a {nontrivial} subspace of $\mathbb{R}^d$, showing prevalence of this phenomenon.

When $d=2$ and $\boldsymbol{M}$ is of full rank,
we consider the $1\times 2$ vector $\boldsymbol{\zeta}^\top \boldsymbol{M}^\top = \boldsymbol{\lambda}^\top \left[\boldsymbol{M}\boldsymbol{M}^\top\right]$. Each $\boldsymbol{\zeta}$ that satisfies the inequality~\eqref{eqn:DPP_eqn} has a unique $\boldsymbol{\lambda}$ that corresponds to it.
As all values of $\boldsymbol{\zeta}^\top\boldsymbol{M}^\top$ span a {nontrivial} subspace of $\mathbb{R}^2$, their corresponding $\boldsymbol{\lambda}$ values span a {nontrivial} subspace of $\mathbb{R}^2$ as well, all of which resulting in DPP. The case of $\boldsymbol{M}$ having rank less than or equal to $1$ corresponds to when $\boldsymbol{\mu}^p$, $\boldsymbol{\mu}^\pi$, and $\boldsymbol{\mu}^L$ lie on the same line or point. It is easy to check using the definition that DPP does not occur in this case.




Denote by ${\cal S} = \{(t_1, t_2, t_3)^\top: (t_1 - t_2)(t_1 - t_3) >0 \} \subset \mathbb{R}^3$. By Definition~\ref{def:dpp},  DPP occurs if and only if
$\boldsymbol{\zeta} \in {\cal S}$.  
When $d \ge 3$ and if $\boldsymbol{\mu}^{\pi}, \boldsymbol{\mu}^L$ and $\boldsymbol{\mu}^p$ are not collinear, the solution of $\boldsymbol{\lambda}$ in the equation $\boldsymbol{\zeta}=\boldsymbol{M}^\top\boldsymbol{\lambda}$, for any $\boldsymbol{\zeta} \in {\cal S}$, always exists and forms a non-trivial subspace of $\mathbb{R}^d$, i.e., set ${\cal D} = \{\boldsymbol{\lambda}: \boldsymbol{\zeta}=\boldsymbol{M}^\top\boldsymbol{\lambda}, \text{ for } \boldsymbol{\zeta} \in {\cal S}\}$ is a non-trivial subspace in $\mathbb{R}^d$.
Any direction $ \boldsymbol{\lambda} \in {\cal D}$, DPP occurs.
Thus, for any linearly independent combination of $\boldsymbol{\mu}^p,\boldsymbol{\mu}^\pi,\boldsymbol{\mu}^L$, DPP would occur for $\boldsymbol{\lambda}$ values in a {nontrivial} subspace of $\mathbb{R}^d$, $d \ge 3$, showing prevalence of this phenomenon.

When $d=2$ and $\boldsymbol{M}$ is of full rank,
we consider 
the equation $ \boldsymbol{M} \boldsymbol{\zeta} =  \boldsymbol{M}\boldsymbol{M}^\top \boldsymbol{\lambda}$, for any $\boldsymbol{\zeta} \in {\cal S}$, which has a unique solution
$\boldsymbol{\lambda} = \left[\boldsymbol{M}\boldsymbol{M}^\top\right]^{-1} \boldsymbol{M} \boldsymbol{\zeta}$. 
In this case, we define ${\cal D} = \{\boldsymbol{\lambda}:  \boldsymbol{M} \boldsymbol{\zeta} =  \boldsymbol{M}\boldsymbol{M}^\top \boldsymbol{\lambda}, \text{ for } \boldsymbol{\zeta} \in {\cal S}\}$. Again, we can see that ${\cal D}$ is a nontrivial subspace in $\mathbb{R}^d$ with $d =2$, showing that DPP is prevalent.
We remark that the case that $\boldsymbol{M}$ is not full rank (with rank less than or equal to one) corresponds to the setting that $\boldsymbol{\mu}^p$, $\boldsymbol{\mu}^\pi$, and $\boldsymbol{\mu}^L$ lie on the same line or are the same point. Only in this case, DPP does not occur.  
}

Although a consequence of probabilistic calculations, at the crux of the DPP lies a puzzle of geometry. And all of the reasoning for prevalence of DPP above can be explained in highly intuitive ways from a geometric perspective in Section~\ref{section:geometry}. 





\section{Conditions for DPP in Exponential-Quadratic Likelihoods}
\label{section:conditions_DPP}

\subsection{Theoretical Results}\label{sec:31}

In this section, we investigate conditions under which DPP occurs for models with multivariate Gaussian priors and exponential-quadratic likelihoods. The latter can be regarded as the asymptotic likelihood in large samples; see
Lemma~\ref{lemma:asymptotic_likelihood}.

We adopt the following notation for the remainder of this section. Let the prior for $\boldsymbol{\theta}\in\mathbb{R}^d$ be 
$\pi(\boldsymbol{\theta} ) = \phi(\boldsymbol{\theta}; \boldsymbol{\mu}^{\pi}, \Sigma^{\pi})$ and the likelihood be proportional to 
$L(\boldsymbol{\theta} | data ) \propto \phi(\boldsymbol{\theta}; \boldsymbol{\mu}^{L}, \Sigma^{L})$,  where $\phi (\cdot; \boldsymbol{\mu}, \Sigma)$ denotes a multivariate Gaussian density with mean $\boldsymbol{\mu}$ and variance $\Sigma$. In this case, it is easy to derive that the posterior distribution of $\boldsymbol{\theta}$ is Gaussian, with mean and variance-covariance matrix denoted by $\boldsymbol{\mu}^{p}$ and $\Sigma^{p}$ respectively. Suppose the parameter of interest is a linear margin $\eta = \boldsymbol{\lambda}^{\top} \boldsymbol{\theta}$, for a given $\boldsymbol{\lambda}\in\mathbb{R}^d$. In what follows, Lemma~\ref{lemma:asymptotic_likelihood} gives two examples of exponential-quadratic likelihoods: one is an exact exponential-quadractic likelihood from independently and identically distributed (i.i.d.) {multivariate} Gaussian observations with unknown mean and known variance, which can be easily adapted to simple linear regression models with unknown regression coefficient and known variance. The other is an asymptotically Gaussian likelihood based on the theory of local asymptotic normality (LAN). We summarize these known results in Lemma~\ref{lemma:asymptotic_likelihood}. This gives concrete examples of exponential-quadractic likelihoods, establishes the notation, and showcases the extent of generality of our analysis.

\begin{lemma} (a) [Gaussian Population]
\label{lemma:asymptotic_likelihood}
Let $d \times 1$ random sample vector ${\bf y}_i\stackrel{\rm i.i.d.}{\sim}\mathcal{N}(\boldsymbol{\theta}, \Lambda)$, $1\leq i\leq n$. Assume that $\Lambda$ is known and $\boldsymbol{\theta}$ is the unknown parameter. 
Then the likelihood is proportional to $L(\boldsymbol{\theta} |{\bf y}) \propto \phi(\boldsymbol{\theta}; \boldsymbol{\mu}^{L}, \Sigma^{L})$ where $\phi$ denotes Gaussian density and $\boldsymbol{\mu}^L = \overline{\bf y}_n = \sum_{i=1}^n {\bf y}_i / n$, $\Sigma^L = \Lambda/n$, 

(b) [Local Asymptotic Normality (LAN)]
Let ${\bf y}_i\stackrel{\rm i.i.d.}{\sim} f(\cdot|\boldsymbol{\theta})$, $1\leq i\leq n$, where $\boldsymbol{\theta}\in\mathbb{R}^d$ is the unknown parameter and $f(\cdot|\boldsymbol{\theta})$ is the density function with regularity conditions given in~\citet[Chapter~6]{le2012asymptotics}. Let $\boldsymbol{\theta}_0$ be the true value of $\boldsymbol{\theta}$. Then in an open neighborhood of $\boldsymbol{\theta}_0$ of radius $O(n^{-1/2})$, with probability converging to $1$ as $n\rightarrow\infty$, the likelihood $L(\boldsymbol{\theta} |{\bf y}) = \prod_{i=1}^n f({\bf y}_i|\boldsymbol{\theta})$, as a function of $\theta$, is proportional to
$L(\boldsymbol{\theta} |{\bf y}) \propto \phi(\boldsymbol{\theta};\boldsymbol{\mu}^L,\Sigma^L)$ for some 
$\boldsymbol{\mu}^L,\Sigma^L$ that only depend on the data and $\boldsymbol{\theta}_0$. 
\end{lemma}

The proof of Lemma~\ref{lemma:asymptotic_likelihood} (a) is trivial.  Lemma~\ref{lemma:asymptotic_likelihood} (b) directly follows from the locally asymptotically quadratic property that is satisfied by a large family of probability distributions~\citep{hajek1972local}. We use the definition given in~\citet[Chapter~6]{le2012asymptotics} to give a proof of Lemma~\ref{lemma:asymptotic_likelihood} in Appendix~\ref{appendix:proof_laq}. \citet{geyer2013asymptotics} also considers quadratic log-likelihoods. Note that, in Lemma~\ref{lemma:asymptotic_likelihood} (b), ${\bf y}_i$ can be either a scalar or vector sample.

Theorem~\ref{theorem:conditionsDPP_gauss} below provides a necessary and sufficient condition for not observing the DPP in exponential-quadratic likelihoods with a multivariate Gaussian prior. 
\begin{theorem}[necessary and sufficient condition for DPP]\label{theorem:conditionsDPP_gauss}
The DPP does not occur if and only if $\Delta_1 = \boldsymbol{\lambda}^{\top}\Sigma^{p}\left(\Sigma^{L}\right)^{-1}\left(\boldsymbol{\mu}^L - \boldsymbol{\mu}^{\pi}\right)$ and $\Delta_2=\boldsymbol{\lambda}^{\top}\Sigma^{p}\left(\Sigma^{\pi}\right)^{-1}\left(\boldsymbol{\mu}^{L} - \boldsymbol{\mu}^{\pi}\right)$ are both positive ($\geq 0$) or both negative ($\leq 0$), where $\Sigma^{p}= \left[\left(\Sigma^{\pi}\right)^{-1} + \left(\Sigma^{L}\right)^{-1}\right]^{-1}$. 
\end{theorem}
See Appendix~\ref{appendix:conditionsDPP_gauss} for a proof of the theorem. Note that $\Delta_1=0$ and $\Delta_2=0$ define two hyper-planes.  
When the samples give a $\boldsymbol{\mu}^L$ that lies on the same side of the two hyperplanes with $\Delta_1\Delta_2 \geq 0$, then DPP does not occur; otherwise, the DPP occurs. This scenario is demonstrated via repeated simulations in Section~\ref{subsection:numerical_results_gaussian}.

Theorem~\ref{theorem:gaussian_conjugate_dpp_prob} below provides the probability that the DPP occurs for the Gaussian population given in Lemma~\ref{lemma:asymptotic_likelihood} (a). This is the sampling probability respect to the true data generating model. 
\begin{theorem}[probability that DPP occurs, {multivariate} Gaussian case]\label{theorem:gaussian_conjugate_dpp_prob}
Using the same notations as in Lemma~\ref{lemma:asymptotic_likelihood} (a), 
the DPP occurs with probability $$P_{\bf{y}_n|\boldsymbol{\theta}}\left(n \boldsymbol{\lambda}^{\top}\Sigma^{p}\Lambda^{-1}\left(\overline{\bf y}_n - \boldsymbol{\mu}^{\pi}\right)\left(\overline{\bf y}_n - \boldsymbol{\mu}^{\pi}\right)^{\top}\left(\Sigma^{\pi}\right)^{-1} \Sigma^{p}\boldsymbol{\lambda}< 0 \right),$$ where the probability $P_{\bf{y}_n|\boldsymbol{\theta}}$ is taken with respect to the true data generating model.  
\end{theorem}

This probability that DPP occurs can be computed using Monte Carlo simulations. For example, in the Gaussian model with a linear contrast of the mean parameters $\eta$, the following data generating model for Gaussian conjugate models is defined in Lemma~\ref{lemma:asymptotic_likelihood} (a): 
    \begin{equation} 
        {\bf y}_i\stackrel{\rm i.i.d.}{\sim} \mathcal{N}(\boldsymbol{\mu}^{o}, \Sigma^{o}), \quad i = 1,\ldots, n; \label{eq:data_generating_model}
    \end{equation}
we have ${\bf z} = \overline{\bf y}_n - \boldsymbol{\mu}^{\pi}\sim \mathcal{N}(\boldsymbol{\mu}^{o} - \boldsymbol{\mu}^{\pi}, \Sigma^{o}/n)$. {Under model~(\ref{eq:data_generating_model}) and for any given $\boldsymbol{\lambda}$ and $(\boldsymbol{\mu}^{\pi}, \boldsymbol{\Sigma}^{\pi})$, w}e can simulate $\bf z$ from this Gaussian repeatedly and count the frequency (probability) of the inequality {in Theorem~\ref{theorem:gaussian_conjugate_dpp_prob}} that defines the DPP holds. Note that here $\Sigma^0$ does not necessarily have to be equal to $\Lambda$, since in practice the data generating model can be different from the model that the data analyst assumes. 

The following result directly follows from Theorem~\ref{theorem:gaussian_conjugate_dpp_prob}, under the data generating model in~\eqref{eq:data_generating_model}. 
\begin{theorem}[{possibility of DPP for {any linear margin}, {multivariate} Gaussian case}] For the Gaussian model in Theorem~\ref{theorem:gaussian_conjugate_dpp_prob} and under the data generating model specified in Equation~\eqref{eq:data_generating_model}, for any $\boldsymbol{\lambda}$, i.e. taking any margin, the probability that DPP occurs is always positive except when $\boldsymbol{\lambda}^{\top}\Sigma^{p}(\Lambda^{-1} - c \left(\Sigma^{\pi}\right)^{-1})=0$ for some positive constant $c$. 
\label{cor1}
\end{theorem}
\begin{remark}
The probability that DPP occurs is equal to zero only when $(\Lambda^{-1} - c \left(\Sigma^{\pi}\right)^{-1})\Sigma^{p}\boldsymbol{\lambda}=0$ has non-zero solutions for $\boldsymbol\lambda$, in other words, the square matrix $(\Lambda^{-1} - c \left(\Sigma^{\pi}\right)^{-1})$ is not invertible for some constant $c$~\citep{schott2010reduced}. A special case of this is when the sample covariance $\Lambda$ is a multiple of the prior covariance $\Sigma^\pi$. 
\end{remark}

{We have so far considered the case where a the parameter of interest  $\eta = \boldsymbol{\lambda}^T \boldsymbol{\theta}$ is {\it any} pre-defined linear margin, with fixed  $\boldsymbol{\lambda}$. Theorem~\ref{cor1} establishes that in the multivariate Gaussian conjugate model, under very mild assumptions on the ranks of the sample and prior covariance matrices, the probability that DPP occurs is positive. 
Theorem~\ref{theorem:gaussian_existence} next establishes that under the same general setting,  with probability one the DPP would occur in {\it some} linear margin of the parameter space.}

\begin{theorem}[{certainty of DPP {in some linear margin}, {multivariate} Gaussian case}]
For the multivariate Gaussian conjugate model given in Lemma~\ref{lemma:asymptotic_likelihood} (a), unless when a non-zero constant $c$ exists such that $c \Lambda =  \Sigma^\pi$, there always exists a nontrivial space for possible $\boldsymbol{\lambda}$ such that the DPP occurs with probability one. Here, the probability is with respect to the joint distribution on $(\boldsymbol{\theta}, {\bf y})$. 
\label{theorem:gaussian_existence}
\end{theorem}
 
A proof of  Theorem~\ref{theorem:gaussian_existence} is given in Appendix~\ref{appendix:gaussian_existence_proof}. This theorem establishes that in the multivariate Gaussian conjugate model, as long as the sample covariance matrix and prior covariance matrix are not multiples of each other, the DPP occurs with probability one in at least one linear margin of the parameter space. 
 
Taken together, the trio of Theorems~\ref{theorem:gaussian_conjugate_dpp_prob},~\ref{cor1} and~\ref{theorem:gaussian_existence} establish the ubiquity of the DPP. They hint that the nature of the phenomenon is an adverse interaction between the geometry of the multi-dimensional parameter space, and probability of the statistical model defined on it. In what follows, we examine several special cases in the Gaussian family of models, and discuss when the DPP does or does not occur in those situations.

\vspace{12pt}

\subsection{Special Cases}
We consider in this subsection several examples that are special cases of Theorem~\ref{theorem:conditionsDPP_gauss}. 
We first focus on two examples in which the parameter of interest is a linear marginal, and then move onto two additional examples of a special case with linear contrast between means in a two-dimensional setup. The purpose of having these concrete, simplified examples is to a) build intuitions on when and how the DPP occurs; b) illustrate how linear contrasts results in DPP in the Gaussian case, thus offering insight on the model that first revealed DPP in \citet{xie2013incorporating} which is a linear contrast; and c) potentially provide guidelines of avoiding them in modeling problems. Furthermore, we lay out the example, Example~\ref{proposition:2dimconstrast_heter_uncor_cov}, for which we illustrate the relationship with the Simpson's paradox in Section~\ref{section:geometry}. {It turns out in the contrasts examples that DPP is not a consequence of parameter dependence among the elements in either the prior or the likelihood specifications. By assuming nonhomogeneous variances in the component dimensions of the parameter is enough to create the unsettling phenomenon.}

Example~\ref{proposotion:generaldim_Gaussian_diagonal_cov} concerns when both the prior and likelihood covariance matrices are diagonal. 


\begin{example}[{diagonal covariances}]\label{proposotion:generaldim_Gaussian_diagonal_cov}
Assume that $\Sigma^{\pi} = {\rm diag}(\sigma_{\pi 1}^2, \ldots, \sigma_{\pi d}^2)$ and $\Sigma^L =  {\rm diag}(\sigma_{L 1}^2, \ldots, \sigma_{L d}^2)$. Define $\omega_j = \frac{\sigma_{\pi j}^{-2}}{\sigma_{\pi j}^{-2} + \sigma_{L j}^{-2}}$ for $j=1,\ldots, d$. In this case, the DPP does not occur if and only if
\begin{align*}
\Delta_1 \Delta_2 = \sum_{i=1}^d \sum_{j=1}^d \lambda_i \lambda_j (1-\omega_i) \omega_j (\mu_i^L - \mu_i^{\pi}) (\mu_j^L - \mu_j^{\pi}) \geq 0,
\end{align*}
where $\Delta_1 = \sum_{i=1}^d \lambda_i (1-\omega_i) (\mu_i^L - \mu_i^{\pi})$ and $\Delta_2 = \sum_{i=1}^d \lambda_i \omega_i (\mu_i^L - \mu_i^{\pi})
$. When $\omega_j =\omega$ for all $j=1,\ldots, d$, $\omega\Delta_1=(1-\omega)\Delta_2$ thus DPP does not occur as long as $\boldsymbol{\mu}^{\pi}\neq \boldsymbol{\mu}^L$ and $\boldsymbol{\lambda}\neq \boldsymbol{0}$.
\end{example}

When $\boldsymbol{\mu}^{\pi}\neq \boldsymbol{\mu}^L$ and $\boldsymbol{\lambda}\neq \boldsymbol{0}$, special cases to the above for which DPP does not occur include {(i)}
when $\sigma_{\pi j}^2 = C \sigma_{L j}^2$ for all $j$ where $C > 0$, that is, the prior and likelihood have the same pattern of heterogeneity; or 
{(ii)}
when $\sigma_{\pi j}^2 = \sigma_{\pi}^2$ and $\sigma_{L j}^2 = \sigma_{L}^2$ for all $j$, that is, the prior and likelihood both have homogeneous, independent dimensions. In other words, when the parameters are orthogonal in both the prior and likelihood, the DPP does not occur when the prior and the likelihood contours are nicely ``aligned'' in the sense of elongated directions/dimensions. 

Example~\ref{proposotion:generaldim_Gaussian_homo_cor} next concerns the situation when both the prior and the likelihood employ homogeneous correlation (or equicorrelation) structure across all dimensions and equal marginal variances. 

\begin{example}[{equicorrelation with homogeneous variances}]\label{proposotion:generaldim_Gaussian_homo_cor}
Assume that $\Sigma^{\pi}= \sigma_{\pi}^2\big[ (1-r) I_d + r \boldsymbol{1}_d \boldsymbol{1}_d^{\top}\big]$ and $\Sigma^{L}= \sigma_{L}^2 \big[(1-\rho) I_d + \rho \boldsymbol{1}_d \boldsymbol{1}_d^{\top}\big]$, where $I_d$ is a diagonal matrix with diagonal elements equal to $1$, $\boldsymbol{1}_d$ is a $d\times 1$ column vector of $1$s, and $-1< r, \rho < 1$; i.e. we have
\begin{equation*}
\Sigma^{\pi} = \sigma_{\pi}^2 \left(\begin{array}{cccc}
1 & r & \cdots & r\\ 
r & 1 & \cdots & r\\ 
\ldots & \ldots & \ldots & \ldots\\ 
r & \cdots & r & 1
\end{array}\right), \Sigma^{L} = \sigma_{L}^2\left(\begin{array}{cccc}
1 & \rho & \cdots & \rho\\ 
\rho & 1 & \cdots & \rho\\ 
\ldots & \ldots & \ldots & \ldots\\ 
\rho & \cdots & \rho & 1
\end{array}\right).
\end{equation*}
Then, DPP does not occur if and only if $\Delta_1\Delta_2 \geq 0$, where
\begin{align*}
\Delta_1
= W_{r\rho} d_{L\pi}^{(1)} - C_{r\rho} d_{L\pi}^{(2)},\quad \Delta_2
= (1-W_{r\rho}) d_{L\pi}^{(1)} + C_{r\rho} d_{L\pi}^{(2)};
\end{align*}
and
$W_{r\rho} = \frac{\sigma_{\pi}^2(1-r)}{\sigma_{\pi}^2(1-r) + \sigma_L^2 (1-\rho)}$, $d_{L\pi}^{(1)} = \boldsymbol{\lambda}^{\top} (\boldsymbol{\mu}^L - \boldsymbol{\mu}^{\pi})$, $d_{L\pi}^{(2)} = \boldsymbol{\lambda}^{\top} \boldsymbol{1}_d \boldsymbol{1}_d^{
\top}(\boldsymbol{\mu}^L - \boldsymbol{\mu}^{\pi})$, and
\begin{align*}
C_{r\rho} &= \frac{\sigma_{\pi}^2 \sigma_L^2 (\rho - r)}{\sigma_{\pi}^2(1-r)+\sigma_L^2 (1-\rho)} \frac{1}{{\sigma_L^2 (\rho d+1-\rho)} + \sigma_{\pi}^2(rd+1-r)}.
\end{align*} 
\end{example}


When $\boldsymbol{\mu}^{\pi}\neq \boldsymbol{\mu}^L$ and $\boldsymbol{\lambda}\neq \boldsymbol{0}$, special cases to the above for which DPP \textit{never} occurs include {(i)} 
when $\rho = r$ (thus $C_{r\rho} = 0$), that is, the prior and likelihood have the same correlation pattern; or {(ii)} 
when $\boldsymbol{\lambda}^{\top} \boldsymbol{1}_d = 0$ (thus $d_{L\pi}^{(2)} = 0$), that is, the prior and likelihood have similar correlation pattern and the parameter of interest is a `contrast'. The special case for which DPP would \textit{always} occur is when $\boldsymbol{\lambda}^{\top}(\boldsymbol{\mu}^L - \boldsymbol{\mu}^{\pi}) = 0$ (thus $d_{L\pi}^{(1)} = 0$) and $\rho\neq r$, $\boldsymbol{\lambda}^{\top} \boldsymbol{1}_d\neq 0$. This corresponds to when the quantity of interest $\eta = \boldsymbol{\lambda}^{\top}\boldsymbol{\theta}$ lies on the direction ($\boldsymbol{\lambda}$) that is orthogonal to the direction of prior-likelihood mean contrast ($\boldsymbol{\mu}^L - \boldsymbol{\mu}^{\pi}$), which is the \textit{farthest away from being a weighted average of the prior mean and the mean given by the data likelihood}. 

\begin{remark}
The situation above when the DPP always occurs is not as significant of a concern as opposed to the seemingly weaker statements of DPP occuring with positive probability. This is because the linear equation that defines this situation, $\boldsymbol{\lambda}^{\top}(\boldsymbol{\mu}^L - \boldsymbol{\mu}^{\pi}) = 0$, actually happens with probability $0$ under the Gaussian conjugate model. Therefore, the more interesting discussions in the paper are related to the cases when DPP occurs with positive probability where the geometry of the prior and likelihood contours ($\Sigma^\pi, \Sigma^L$ in the Gaussian model) plays an important role.
\end{remark}

We now examine linear contrasts of the two-dimensional posterior mean, {that is, $\boldsymbol{\lambda} = (1, -1)^\top$}, special cases of the previous examples to gain more intuition. Example~\ref{proposition:2dimconstrast_hom_cov} shows that the DPP does not occur as long as the two component dimensions of the parameter {have the same variance within the prior and the likelihood specifications}, regardless of correlation structure. On the contrary, Example~\ref{proposition:2dimconstrast_heter_uncor_cov} shows that when the variances of the two dimensions differ, it creates the possibility for DPP even if the two dimensions are independent within both the prior and likelihood. In what follows, $\Delta_\pi = (\mu_1^{\pi} - \mu_2^{\pi})$, and $\Delta_L = (\mu_1^{L} - \mu_2^{L})$.

\begin{example}[two-dimensional contrast, homogeneous variance] \label{proposition:2dimconstrast_hom_cov}
If $\Sigma^{\pi} = \sigma_{\pi}^2 \left(\begin{array}{cc}
1 & r\\ r & 1
\end{array}\right)$, $\Sigma^{L} = \sigma_{L}^2\left(\begin{array}{cc}
1 & \rho\\ \rho & 1
\end{array}\right)$, where $-1< r, \rho< 1$, then for $\boldsymbol{\lambda} = (1, -1)^\top$, the DPP does not occur. 
\end{example}

A special case of Example~\ref{proposition:2dimconstrast_hom_cov} is when $r=\rho = 0$. The posterior distribution for $\theta_1 - \theta_2$ is $\mathcal{N}\left( w\Delta_\pi + (1-w) \Delta_L, 2\left[\sigma_{\pi}^{-2} + \sigma_{L}^{-2}\right]^{-1} \right)$, where $w = \frac{\sigma_{\pi}^{-2}}{\sigma_{\pi}^{-2} + \sigma_{L}^{-2}}\in (0, 1)$. Note that the posterior mean $w\Delta_\pi + (1-w) \Delta_L$ must lie between $\Delta_\pi$ and $\Delta_L$, regardless which one is larger. Another special case is when only $r=0$, i.e. for correlated parameters' likelihood, we set an independent prior; or similarly when only $\rho = 0$, i.e. for uncorrelated parameters' likelihood, we set a correlated prior. Again, we can write the posterior mean for the contrast as a convex combination of the prior contrast $\Delta_{\pi}$ and the MLE $\Delta_L$. Thus the DPP does not occur. The detailed result and proof are given in Appendix~\ref{appendix:statement_proof_gauss_homo_corollary}. Example~\ref{proposition:2dimconstrast_hom_cov} shows that in practice, if we can make the the marginal variances of the parameters in both the prior and likelihood close to being homogeneous, DPP could be mitigated or even avoided. In fact, the homogeneity of marginal variances is a nice property to have not only for avoiding the DPP, but also for the efficiency of computational algorithms in Bayesian inference, which we discuss in more details in Section~\ref{section:summary}.

In contrast to Example~\ref{proposition:2dimconstrast_hom_cov}, Example~\ref{proposition:2dimconstrast_heter_uncor_cov} gives the condition under which DPP occurs when the two component dimensions of the parameter are uncorrelated, but the marginal variances are not the same. 

\begin{example}[two-dimensional contrast, {heterogeneous variance}]\label{proposition:2dimconstrast_heter_uncor_cov}
Let $\Sigma^{\pi} =  \left(\begin{array}{cc}
s_{\pi}^2 & 0\\ 0 & \sigma_{\pi}^2
\end{array}\right)$, $\Sigma^{L} = \left(\begin{array}{cc}
s_L^2 &  0\\ 0 & \sigma_{L}^2
\end{array}\right)$, $\boldsymbol{\lambda}= (1, -1)^\top$, and denote
\begin{equation} \label{eq:2contrast-weights}
	w^{s} = \frac{s_{\pi}^{-2}}{s_{\pi}^{-2} + s_{L}^{-2}}, \quad w^{\sigma} = \frac{\sigma_{\pi}^{-2}}{\sigma_{\pi}^{-2} + \sigma_{L}^{-2}}.
\end{equation}
Then the posterior mean for $\theta_1-\theta_2$ is $\Delta^* = w^s \mu_1^{\pi} - w^{\sigma}\mu_2^{\pi} + (1-w^s) \mu_1^{L} - (1-w^{\sigma})\mu_2^{L}$. 
\begin{enumerate}
    \item[(i)] When $w^{\sigma} = w^s$, i.e. the relative curvature between the prior and likelihood is the same for the two dimensions, $\min\{\Delta_{\pi}, \Delta_{L}\}\leq\Delta^*\leq \max\{\Delta_{\pi}, \Delta_{L}\}$ always holds and the equality holds if and only if $\Delta_L=\Delta_{\pi}$. In the special case when $\Delta_L=\Delta_{\pi}$, we have $\Delta^*=\Delta_L=\Delta_{\pi}$, which is perfect alignment of prior/MLE/posterior. Thus the DPP does not occur in this case. 
    \item[(ii)] When $w^{\sigma} \neq w^s$, without loss of generality, we assume that $w^{\sigma} > w^s$, then DPP occurs if and only if $\mu_2^L\neq \mu_2^{\pi}$ and
\begin{equation}\label{eq:2contrast}
\frac{1-w^{\sigma}}{1-w^s} <\frac{\mu_1^L - \mu_1^{\pi}}{\mu_2^L - \mu_2^{\pi}} < \frac{w^{\sigma}}{w^s}.
\end{equation}
\end{enumerate}
\end{example}

Example~\ref{proposition:2dimconstrast_heter_uncor_cov} sends a somewhat surprising message, as compared to the commonly perceived understandings of ``difficult geometry'' of the likelihood and prior misalignment. As it turns out, DPP is not a consequence of parameter dependence in either the prior or the likelihood specifications. Just by assuming nonhomogeneous variances in the component dimensions of the parameter is enough to create the unsettling phenomenon. The geometry behind Proposition 4 is the subject of detailed analysis in Section~\ref{section:geometry}.

\subsection{Numerical Results}
\label{subsection:numerical_results_gaussian}
Numerical results based on repeated simulations under various multivariate Gaussian models corresponding to {the heterogeneous variance case with and without correlation structures (the latter corresponds to Example~\ref{proposition:2dimconstrast_heter_uncor_cov})} 
are collected in Figure~\ref{fig:gauss_dpp}. 
The parameter of interest is $\eta = \theta_1 - \theta_2$, the difference of the two Gaussian marginal means. In columns 1 and 3, the prior mean for $\boldsymbol{\theta}$ is $(-1, 1)^{\top}$ and in columns 2 and 4, the prior mean for $\boldsymbol{\theta}$ is $(0, 0)^{\top}$. In columns 1 and 2, $\Lambda = \left(\begin{array}{cc}
        7 & 2 \\
        2 & 1
    \end{array}\right), \Sigma_{\pi} = \left(\begin{array}{cc}
        5 & 4 \\
        4 & 4
    \end{array}\right)$; and in columns 3 and 4, $\Lambda = \left(\begin{array}{cc}
        7 & 0 \\
        0 & 1
    \end{array}\right), \Sigma_{\pi} = \left(\begin{array}{cc}
        5 & 0 \\
        0 & 4
    \end{array}\right)$. The data generating models are given by Gaussian with mean $(0,0)^\top$ and variance-covariance $\Lambda$. Monte Carlo estimates of the probabilities of DPP under each model are given. We can see that the DPP 
    {is gradually mitigated} as we increase the sample size, although at a slower rate for some models than others. For the examples shown here, models with uncorrelated parameter components (in both the likelihood and the prior) seem less prone to DPP than models with highly correlated dimensions. However, DPP is not eliminated in these cases, and the extent of reduction is a function of the parameter values used for the simulations shown here. In Example~\ref{proposition:2dimconstrast_heter_uncor_cov}, heterogeneous variances with independent dimensions is shown to be related to the DPP. This example shows that heterogeneous variances plus correlation among parameters make the situation even worse. Models with priors not geometrically aligned with the likelihood, e.g. with misaligned prior mean values and MLEs, heterogeneous marginal variances and/or correlation structure, are more likely to suffer from DPP than otherwise.
     
\begin{figure}[tbph]
    \centering
    \includegraphics[width=\textwidth]{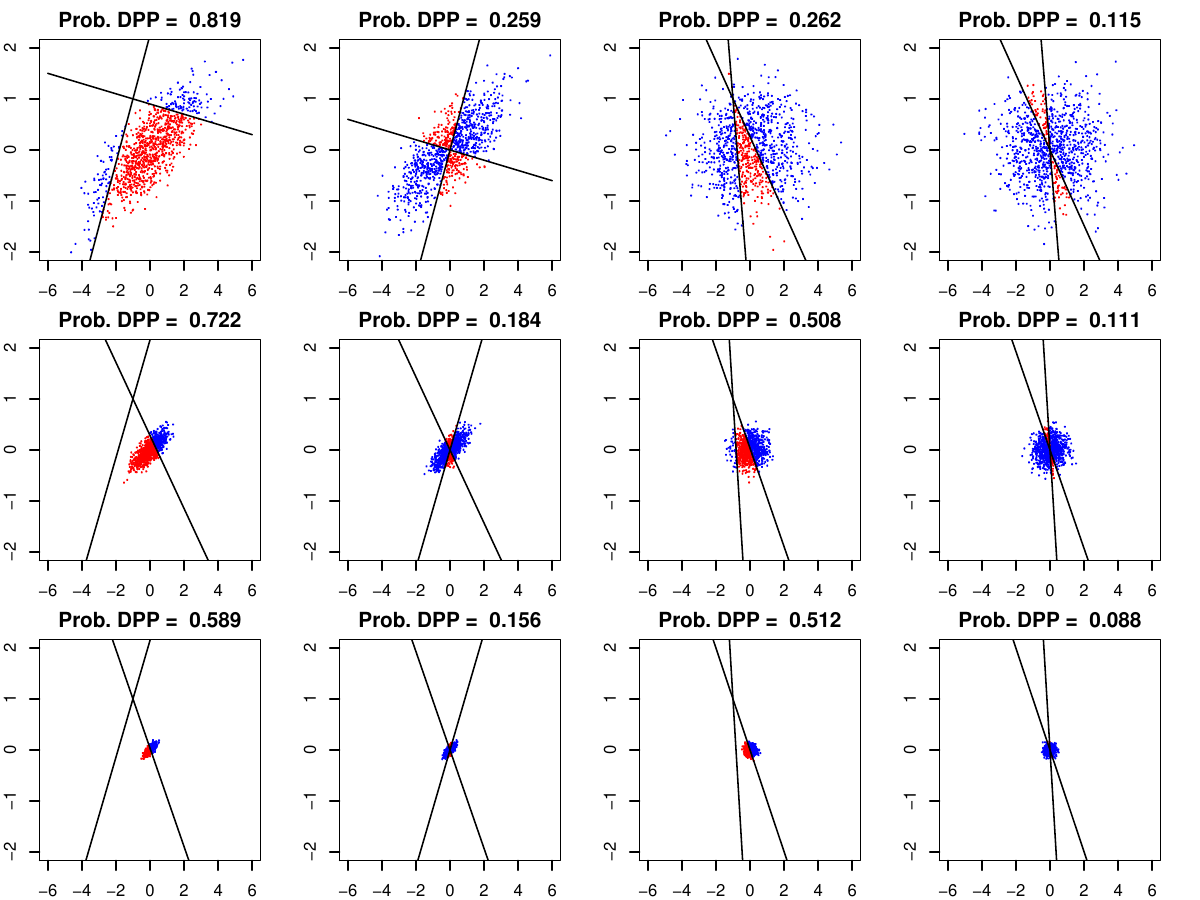}
    \caption{DPP in bivariate Gaussian models for linear contrast. Each subplot contains $1,000$ independently simulated datasets from Lemma~\ref{lemma:asymptotic_likelihood}(a), with $n=3$, $30$, and $300$ samples per dataset for each of the three rows. Each dataset is represented by a point whose x and y coordinates are the sample averages of the first and second dimensions respectively. The red dots are data occurrences for which DPP occurs, and the blue dots are those for which DPP does not occur. Separating the two are hyperplanes $\Delta_1=0$ and $\Delta_2=0$ defined in Theorem~\ref{theorem:conditionsDPP_gauss}.} 
    \label{fig:gauss_dpp}
\end{figure}


\section{The geometry of DPP and relation to Simpson's paradox}\label{section:geometry}

\subsection{An Illustration of DPP Geometry in 2-Dim Case}\label{subsec:geometry}

In this section, we take a closer look at the geometry behind DPP, and illustrate its connection with Simpson's paradox, one that occurs due to  
inconsistently aggregating sources of conditional information. For simplicity, the analysis below focuses on the scenario described in Example~\ref{proposition:2dimconstrast_heter_uncor_cov} with $d = 2$, where we assume a bivariate Gaussian conjugate model for the Bayesian assessment of the treatment efficacy of a drug in relation to a placebo. Here, the efficacy of either treatments is measured by a real number. Both the prior and the sampling distribution have independent and heterogeneous covariance structures. The inferential target is again the posterior linear contrast between the efficacy of the drug and placebo treatments. Thus, $\eta = 
\boldsymbol{\lambda} \boldsymbol{\theta} = (1, -1) \boldsymbol{\theta} = \theta_1 - \theta_2.$  

For the purpose of illustration, assume the prior mean $\boldsymbol{\mu}^{\pi} = (0.25, 0.45)$ and the MLE $\boldsymbol{\mu}^L = (1.10, 1.15)$, with respective diagonal covariance matrices  ${\Sigma^{\pi}}$ and ${\Sigma^{L}}$. The model is depicted in Figure~\ref{fig:simpson_dpp}. Note that the MLE is greater than the prior mean element-wise, that is, $\boldsymbol{\mu}^L$ is to the northeast of $\boldsymbol{\mu}^{\pi}$ in the plot. Denote $\boldsymbol{\mu}^{p}$ the posterior mean. 
Since both the prior and likelihood covariances are diagonal, the light blue rectangle with $\boldsymbol{\mu}^{\pi}$ and $\boldsymbol{\mu}^{L}$ as vertices is the region in which $\boldsymbol{\mu}^{p}$ could take value. Three lines of slope 1 pass through $\boldsymbol{\mu}^{\pi}$, $\boldsymbol{\mu}^{L}$ and $\boldsymbol{\mu}^{p}$, and intersect the $y$-axis at $A$, $B$ and $C$ respectively. The $y$-coordinates of the three intersections are respectively the prior, likelihood, and posterior linear contrasts, that is, $A = (0, \boldsymbol{\lambda}^{\top}\boldsymbol{\mu}^{\pi})$, $B= (0, \boldsymbol{\lambda}^{\top}\boldsymbol{\mu}^{L})$, and $C = (0, \boldsymbol{\lambda}^{\top}\boldsymbol{\mu}^{p})$, where $\boldsymbol{\lambda} = (1, -1)^\top$. 

By Definition~\ref{def:dpp}, DPP occurs if $C$ falls outside the closed interval between $A$ and $B$. Equivalently stated, the occurrence of DPP can be determined by examining the location of $\boldsymbol{\mu}^{p}$ relative to the dark blue parallelogram sandwiched between the two lines that pass through $\boldsymbol{\mu}^{\pi}$ and $A
 $, as well as $\boldsymbol{\mu}^{L}$ and $B$. DPP occurs if $\boldsymbol{\mu}^{p}$ falls within the light blue rectangle but outside the parallelogram, and it does not occur if $\boldsymbol{\mu}^{p}$ falls within parallelogram.

Having fixed $\boldsymbol{\mu}^{\pi}$ and $\boldsymbol{\mu}^{L}$, the location of $\boldsymbol{\mu}^{p}$ is a function of the prior and likelihood covariances $\Sigma^{\pi}$ and $\Sigma^{L}$. The specific values depicted in Figure~\ref{fig:simpson_dpp} are 
${\Sigma^{\pi}}= \text{diag}(3, 9)$
and 
${\Sigma^{L}}=\text{diag}(7,3)$. The posterior mean is then $\boldsymbol{\mu}^{p} = (0.505, 0.975)$, and posterior covariance 
$\Sigma^{p} = \text{diag}(2.1, 2.25)$. The three covariances are illustrated by their respective concentration ellipses around $\boldsymbol{\mu}^{\pi}$, $\boldsymbol{\mu}^{L}$ and $\boldsymbol{\mu}^{p}$. Notice that $\boldsymbol{\mu}^{p}$ falls within the light blue region outside the dark blue band. As a consequence, the posterior contrast ($+0.47$) is larger than both the prior ($+0.2$) and likelihood contrasts ($+0.05$), suggesting that the efficacy of the drug is assessed to be more than the placebo {\it a posteriori}, at a scale larger than that of either the prior or the data alone. 


To understand the geometry of Example~\ref{proposition:2dimconstrast_heter_uncor_cov}, define $\alpha, \beta \in (-\pi/2, \pi/2)$ such that 
\begin{equation*}
\tan\alpha=\frac{w^{\sigma}}{w^{s}}\cdot\frac{\mu_{2}^{L}-\mu_{2}^{\pi}}{\mu_{1}^{L}-\mu_{1}^{\pi}},\quad\tan\beta=\frac{1-w^{\sigma}}{1-w^{s}}\cdot\frac{\mu_{2}^{L}-\mu_{2}^{\pi}}{\mu_{1}^{L}-\mu_{1}^{\pi}}.
\end{equation*}
The two angles $\alpha$ and $\beta$ are annotated in Figure~\ref{fig:simpson_dpp}. Equation (\ref{eq:2contrast}) can be re-expressed as
\begin{equation}\label{eq:angle-alpha-beta}
\tan\beta<1<\tan\alpha.
\end{equation}
That is, given that $\mu_{2}^{L}\neq\mu_{2}^{\pi}$, DPP occurs if and only if $\beta<\pi/4<\alpha$. This happens precisely when $\boldsymbol{\mu}^{p}$ sits to the left of the line that passes through $\boldsymbol{\mu}^{\pi}$ and $A$. For the specific values of $\Sigma^{\pi}$ and $\Sigma^{L}$ in this example, the weights satisfy $w^{\sigma} > w^{s}$ with values equal to $0.7$ and $0.25$ respectively. Should it be the case that $w^{\sigma}<w^{s}$, the same argument applies once the roles of $\boldsymbol{\mu}^{\pi}$ and $\boldsymbol{\mu}^{L}$ are flipped. The necessary and sufficient condition for DPP to occur is then $\alpha<\pi/4<\beta$, or equivalently, for $\boldsymbol{\mu}^{p}$ to sit to the right of the line that passes through $\boldsymbol{\mu}^{L}$ and $B$.




\begin{figure}[tbph]
    \centering
    \includegraphics[width=0.7\textwidth]{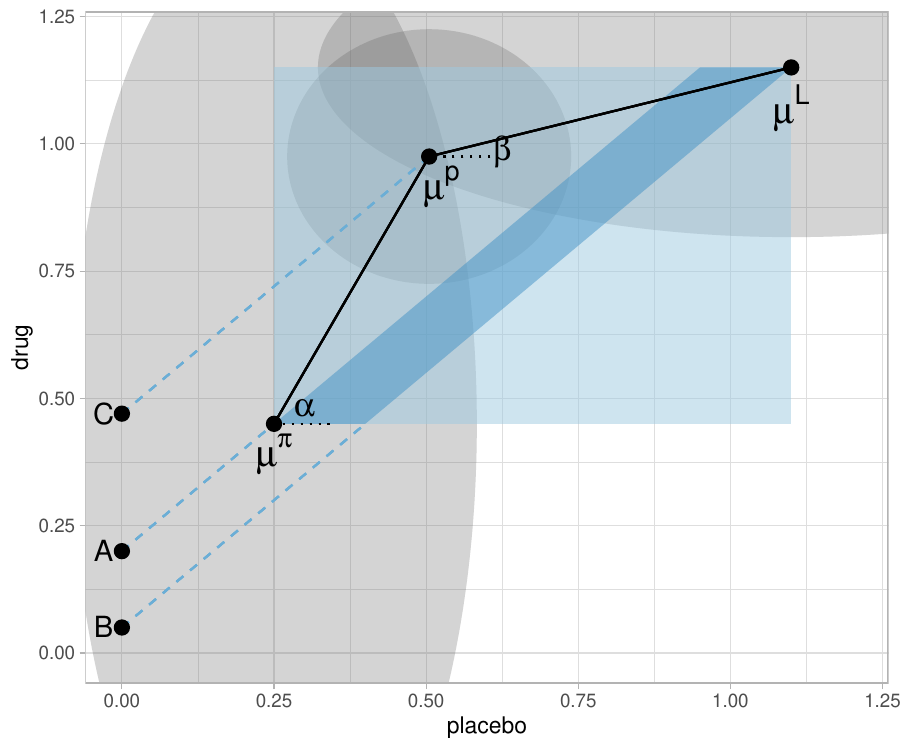}
    \caption{The geometry of DPP: posterior drug efficacy (C) exceeds the range spanned by those assessed from the prior (A) and the data (B). That is because the posterior mean ($\mu^{p}$) is only an {\it element-wise} convex combination of the prior mean ($\mu^{\pi}$) and the MLE ($\mu^{L}$),  but itself is not collinear with them. In gray are concentration ellipses of the covariance matrices ${\Sigma^{\pi}}$,  ${\Sigma^{L}}$ and $\Sigma^{p}$.}
    \label{fig:simpson_dpp}
\end{figure}

For conjugate normal models, the posterior mean $\boldsymbol{\mu}^{p}$ is an {\it element-wise} convex combination of the prior mean $\boldsymbol{\mu}^{\pi}$ and the MLE $\boldsymbol{\mu}^{L}$. That is, each dimension of $\boldsymbol{\mu}^{p}$ is a convex combination of the corresponding dimensions of $\boldsymbol{\mu}^{\pi}$ and $\boldsymbol{\mu}^{L}$, with weights determined by the prior and the sampling distribution covariances. If the weights applied to $\boldsymbol{\mu}^{\pi}$ and $\boldsymbol{\mu}^{L}$ are not balanced across dimensions, the resulting posterior mean $\boldsymbol{\mu}^{p}$ may not be an {\it overall} convex combination of $\boldsymbol{\mu}^{\pi}$ and $\boldsymbol{\mu}^{L}$, which is to say that it may not be collinear with $\boldsymbol{\mu}^{\pi}$ and $\boldsymbol{\mu}^{L}$. Indeed, when the weights are heavily imbalanced, $\boldsymbol{\mu}^{p}$ can be far from collinear with $\boldsymbol{\mu}^{\pi}$ and $\boldsymbol{\mu}^{L}$, much so that it creates ample triangularization among the three quantities for a collection of marginal directions to render the projection of  $\boldsymbol{\mu}^{p}$ outside the range of those of $\boldsymbol{\mu}^{\pi}$ and $\boldsymbol{\mu}^{L}$. Referring again to Figure~\ref{fig:simpson_dpp}, any value of $\boldsymbol{\mu}^{p}$ outside the dark blue parallelogram is considered far from collinear with $\boldsymbol{\mu}^{\pi}$ and $\boldsymbol{\mu}^{L}$, giving rise to the DPP.

To put this formally, let $\boldsymbol{\lambda}$ denote the linear margin of interest, and consider the two angles it forms with $(\boldsymbol{\mu}^p-\boldsymbol{\mu}^{L})$  and $(\boldsymbol{\mu}^p-\boldsymbol{\mu}^{\pi})$ respectively, namely $\phi^{L}$ and $\phi^{\pi}$ such that
\begin{equation}\label{eq:angles}
\cos(\phi^{L}) = \frac{\boldsymbol{\lambda}^{\top}(\boldsymbol{\mu}^p-\boldsymbol{\mu}^{L})}{|\boldsymbol{\lambda}|\cdot|\boldsymbol{\mu}^p-\boldsymbol{\mu}^{L}|}, \qquad     \cos(\phi^{\pi}) = \frac{\boldsymbol{\lambda}^{\top}(\boldsymbol{\mu}^p-\boldsymbol{\mu}^{\pi})}{|\boldsymbol{\lambda}|\cdot|\boldsymbol{\mu}^p-\boldsymbol{\mu}^{\pi}|}.
\end{equation}

By Definition~\ref{def:dpp}, the DPP occurs if and only if both cosine quantities in \eqref{eq:angles} are positive or negative. If $\boldsymbol{\mu}^{p}$ is not exactly collinear with $\boldsymbol{\mu}^{\pi}$ and $\boldsymbol{\mu}^{L}$, the marginal direction orthogonal to $(\boldsymbol{\mu}^{L} - \boldsymbol{\mu}^{\pi})$ is always vulnerable to the DPP. Indeed, any departure in $\boldsymbol{\mu}^{p}$ from the convex combination of $\boldsymbol{\mu}^{\pi}$ and $\boldsymbol{\mu}^{L}$ can be picked up by the marginal direction orthogonal to the difference of the latter two, however slight the departure may be. In addition, the neighborhood of marginal directions whose polar angles are between $\beta + \pi/4$ and $\alpha + \pi/4$ are also vulnerable to the DPP. As long as \eqref{eq:angle-alpha-beta} holds, this neighborhood is nonempty.

The geometry described here is not limited to two-dimensional situations. The same intuition applies when the Bayesian model of concern invokes a parameter space of higher dimensions. In fact, the higher the dimension of the parameter space, the more ``prevalent'' the DPP in the sense that the nonempty neighborhood of marginal directions $\boldsymbol{\lambda}$ that can result in a DPP is also of higher dimension, and can be harder to avoid, see Section~\ref{appendix:geometry_dpp_general}. 

\subsection{Connections with Simpson's Paradox}\label{subsec:simpson}

The DPP is keenly related to the Simpson's paradox~\citep{blyth1972simpson}, another puzzling phenomenon that occurs when the marginal expectation of a random variable apparently takes value outside the range of the conditional expectations of the same variable from which it is aggregated. Simpson's paradox is a consequence of incoherent marginalization: sources of conditional information were aggregated against different, as opposed to the same, marginal distributions of the conditioning variable. When the difference is substantial, the marginal expectation may appear out of range, which is otherwise mathematically impossible had the marginalization been done coherently.


The Simpson's paradox has been studied extensively in the statistics literature, in particular in the context of causal inference from observational studies. The paradox is suspect when there may exist unmeasured confounding variables that introduce systematic discrepancies in the marginalization schemes of a quantity of interest. While not a mathematical anomaly, the presence of the Simpson's paradox undermines the trustworthiness of the inference and the causal implication it may have on external generalizations \citep[][chapter 6]{pearl2009causality}; see also \cite{armistead2014resurrecting,christensen2014comment,liu2014comment,pearl2014understanding} for a recent discussion. For this reason, sensitivity
analysis in causal inference seeks to establish deterministic bounds to exclude scenarios that are essentially equivalent to the Simpson's
paradox \citep{ding2016sensitivity}.

In what follows, we make precise the analogy between Simpson's paradox and the DPP, using the same hypothetical example of a drug efficacy study as set up in Section~\ref{subsec:geometry}. We will see that when the prior and posterior means in the conjugate Gaussian model are regarded as two point estimators of the quantity of interest to be aggregated, as it has been the case with our investigation, DPP is precisely a manifestation of the Simpson's paradox. 

Suppose the clinical study of the efficacy of a drug against a placebo consists of two phases. A pilot study is conducted first, followed by a full clinical trial. In the practical design of large-scale experiments, pilot studies are often performed to provide preliminary guidance for the subsequent full experiment, and are usually smaller and more cost-effective. If the Bayesian approach is employed to analyze all available observations, information supplied by the pilot study is naturally understood as the source of prior information, relative to the full experiment that supplies the data likelihood.

Let $T_1$ and $T_2$ be the estimators of the drug and placebo efficacy respectively. Let $Z$ be the indicator variable of whether the observation is made through the pilot study ($Z=0$), which corresponds to the prior, or the full clinical trial ($Z=1$) which corresponds to the likelihood. Write
\begin{align*}
\mathbb{E}\left(T_{1}\mid Z=0\right)=\mu_{1}^{\pi}, & \quad  E\left(T_{2}\mid Z=0\right)=\mu_{2}^{\pi},\\
\mathbb{E}\left(T_{1}\mid Z=1\right)=\mu_{1}^{L}, & \quad  E\left(T_{2}\mid Z=1\right)=\mu_{2}^{L},
\end{align*}
with the understanding that the expectations are each taken with respect to a distinct and independent sample drawn from a (possibly finite) population.

Let $\boldsymbol{w}=\left(w,1-w\right)^{\top}$ be the sample marginal distribution of $Z$, that is the fraction of subjects assigned to the pilot study versus the clinical trial. Write $\boldsymbol{\mu}_{1}=\left(\mu_{1}^{\pi},\mu_{1}^{L}\right)$ and $\boldsymbol{\mu}_{2}=\left(\mu_{2}^{\pi},\mu_{2}^{L}\right)$. If $Z$ is independent of $T_{1}$ and $T_{2}$, the marginal expected linear contrast can be written as
\begin{equation*}
\boldsymbol{w}^{\top}\boldsymbol{\mu}_{1}-\boldsymbol{w}^{\top}\boldsymbol{\mu}_{2} =  \sum_{z}E\left(T_{1}-T_{2}\mid Z=z\right)P\left(Z=z\right)	= E\left(T_{1}-T_{2}\right).
\end{equation*}
As $w$ varies in $[0,1]$, it is guaranteed that
\begin{equation}\label{eq:marginal-bound}
\min\left(\mu_{1}^{\pi}-\mu_{2}^{\pi},\mu_{1}^{L}-\mu_{2}^{L}\right)\le\boldsymbol{w}^{\top}\boldsymbol{\mu}_{1}-\boldsymbol{w}^{\top}\boldsymbol{\mu}_{2}\le\max\left(\mu_{1}^{\pi}-\mu_{2}^{\pi},\mu_{1}^{L}-\mu_{2}^{L}\right).	
\end{equation}
That is, the marginal expected linear contrast is bounded within range of the conditional expected linear contrasts from the pilot study and the clinical trial. In other words, Simpson's paradox does not occur regardless of $\boldsymbol{w}$. However, if $Z$ is not independent of $T_{1}$ and $T_{2}$, that is if the assignment probabilities to the pilot study versus the clinical trial depend on the outcomes, the guarantee in (\ref{eq:marginal-bound}) does not hold. In particular, if $Z$ possesses two distinct marginal distributions, one pertinent to either the drug ($T_{1}$) or the placebo ($T_{2}$), and are respectively  
\begin{equation}\label{eq:w1w2}
\boldsymbol{w}_{1}=\left(w^{s},1-w^{s}\right)^{\top}\quad\text{and}	\quad \boldsymbol{w}_{2}=\left(w^{\sigma},1-w^{\sigma}\right)^{\top},
\end{equation}
then the ``marginal expected linear contrast'' of the drug's efficacy is written as 
\begin{equation}\label{eq:marginal-bound2}
\boldsymbol{w}_{1}^{\top}\boldsymbol{\mu}_{1}-\boldsymbol{w}_{2}^{\top}\boldsymbol{\mu}_{2} =  \lambda^{\top}\boldsymbol{\mu}^{p}.
\end{equation}
The phrase ``marginal expected linear contrast'' here is in quotes, as the marginalization of $T_1$ and $T_2$ endorsed different marginal distributions of $Z$, hence the result is incoherent for comparison purposes. In this case, Simpson's paradox is said to occur whenever
\begin{equation*}
\boldsymbol{w}_{1}^{\top}\boldsymbol{\mu}_{1}-\boldsymbol{w}_{2}^{\top}\boldsymbol{\mu}_{2} <  \min\left(\mu_{1}^{\pi}-\mu_{2}^{\pi},\mu_{1}^{L}-\mu_{2}^{L}\right), \; \text{or}\;\;  \boldsymbol{w}_{1}^{\top}\boldsymbol{\mu}_{1}-\boldsymbol{w}_{2}^{\top}\boldsymbol{\mu}_{2} > \max\left(\mu_{1}^{\pi}-\mu_{2}^{\pi},\mu_{1}^{L}-\mu_{2}^{L}\right),	
\end{equation*}
which coincides with the definition of DPP in \eqref{eqn:DPP_eqn}. By the geometric analysis in Section~\ref{subsec:geometry}, we know that the case illustrated here is another instance of Example~\ref{proposition:2dimconstrast_heter_uncor_cov}, where $\boldsymbol{w}_{1}$ and $\boldsymbol{w}_{2}$ as defined in \eqref{eq:w1w2} take values according to the respective posterior variance component coefficients of \eqref{eq:2contrast-weights}, and \eqref{eq:marginal-bound2} is precisely the posterior expected linear contrast with respect to the independent heterogeneous variance model. The Simpson's paradox occurs here, precisely when the DDP occurs there.

\subsection{Geometry of DPP for General Parameter Dimensions}
\label{appendix:geometry_dpp_general}
We illustrate the geometry of DPP for a general parameter dimension $d$ with Figure~\ref{fig:DPP_general}. We visualize the vectors ($\boldsymbol{\lambda}$, $\boldsymbol{\mu}^p -\boldsymbol{\mu}^\pi$, $\boldsymbol{\mu}^p-\boldsymbol{\mu}^L$) and hyper-planes in a 3-d plot to explicitly demonstrate the DPP in higher dimensions. The coordinate orientation and vectors are chosen such that our figure reflects general situations instead of special cases. Denote the linear space spanned by $(\boldsymbol{\mu}^p-\boldsymbol{\mu}^{\pi},\boldsymbol{\mu}^p-\boldsymbol{\mu}^{L})$  as $\mathcal{S}^{\pi L}$, and the projection of $\boldsymbol{\lambda}$ to $\mathcal{S}^{\pi L}$ as $\boldsymbol{\lambda}^{\rm proj}$. We only need to consider $\boldsymbol{\lambda}^{\rm proj}$, since the remainder $\boldsymbol{\lambda}-\boldsymbol{\lambda}^{\rm proj}$ is orthogonal to any vectors in $\mathcal{S}^{\pi L}$ thus does not contribute to either $\cos(\phi^\pi)$ or $\cos(\phi^L)$, where $\phi^\pi$ and $\phi^L$ are defined in~\eqref{eq:angles}, i.e. the angles between $\boldsymbol{\lambda}^{\rm proj}$ (or $\boldsymbol{\lambda}$ equivalently) and $\boldsymbol{\mu}^p -\boldsymbol{\mu}^\pi$,{ $\boldsymbol{\mu}^p - \boldsymbol{\mu}^L$} respectively . Without loss of generality, assume that the angle between the two vectors $\boldsymbol{\mu}^p-\boldsymbol{\mu}^{\pi}$ and $\boldsymbol{\mu}^p-\boldsymbol{\mu}^{L}$ is between $0$ and $\pi$. 
Let $\mathcal{P}^{\pi}$ and $\mathcal{P}^L$ denote subspaces (lines in the 3-d plot) within $\mathcal{S}^{\pi L}$ that are orthogonal to $\boldsymbol{\mu}^p-\boldsymbol{\mu}^{\pi}$ and $\boldsymbol{\mu}^p-\boldsymbol{\mu}^{L}$, respectively.  The half-space of $\mathcal{S}^{\pi L}$, separated by $\mathcal{P}^L$ and containing $\boldsymbol{\mu}^p-\boldsymbol{\mu}^{L}$, consists of vectors that intersect with  $\boldsymbol{\mu}^p-\boldsymbol{\mu}^{L}$ at an acute angle between $[0, \pi/2)$. Similarly, the half-space of $\mathcal{S}^{\pi L}$, separated by $\mathcal{P}^\pi$ and containing $\boldsymbol{\mu}^p-\boldsymbol{\mu}^{\pi}$, consists of vectors that intersect with $\boldsymbol{\mu}^p-\boldsymbol{\mu}^{\pi}$ at an acute angle as well. Note that $\pi/2$ is a critical angle where the cosine function changes from a positive to a negative sign. Thus, if $\boldsymbol{\lambda}^{\rm proj}$ either lies within the intersection of the two  half-spaces respectively separated by $\mathcal{P}^L$ and $\mathcal{P}^{\pi}$, or that it lies within the intersection of their complement half-spaces, then $\phi^{\pi},\phi^L$ are either both acute or both obtuse, with $\cos(\phi^{\pi})\cos(\phi^L)>0$, and the DPP would occur. And this region where $\boldsymbol{\lambda}^{\rm proj}$ can take values to result in DPP is given by the blue shaded region ($\mathcal{S}^{\pi L}$) except the grey shaded region in Figure~\ref{fig:DPP_general}. On the other hand, if  $\boldsymbol{\lambda}^{\rm proj}$ lies within the relative complement of the two half-spaces, the DPP would not occur; and this region where $\boldsymbol{\lambda}^{\rm proj}$ can take values are given as the grey shaded region in Figure~\ref{fig:DPP_general}. From here, we see that the choice of $\boldsymbol{\lambda}$ that results in DPP is a nontrivial subspace in $\mathbb{R}^{d}$, affirming the geometric prevalence of the DPP as previously established.

Furthermore, we note that from the reasoning above, especially Figure~\ref{fig:DPP_general}, if the angle between $\boldsymbol{\mu}^p -\boldsymbol{\mu}^\pi$ and $\boldsymbol{\mu}^p -\boldsymbol{\mu}^L$ is between $(0, \pi/2)$, then the region where $\boldsymbol{\lambda}$ can take values to result in DPP is larger than its complement in $\mathbb{R}^d$. This is evident from the fact that the grey shaded region is smaller than its complement in $\mathcal{S}^{\pi L}$ (the blue shaded region) in Figure~\ref{fig:DPP_general}; when the angle between $\boldsymbol{\mu}^p -\boldsymbol{\mu}^\pi$ and $\boldsymbol{\mu}^p -\boldsymbol{\mu}^L$ is less than $\pi/2$. This situation corresponds to when the joint posterior is demonstrating some ``deviation'' from both the prior and likelihood towards certain direction, thus the DPP on a marginal posterior is more prevalent. However, on the contrary, if the angle between $\boldsymbol{\mu}^p -\boldsymbol{\mu}^\pi$ and $\boldsymbol{\mu}^p -\boldsymbol{\mu}^L$ is between $\pi/2$ and $\pi$, we know that the joint posterior is closer to the case when the posterior is ``in the middle'' of the prior and the likelihood; thus in this case, the DPP on the marginal posterior is less prevalent. But in either case, the $\boldsymbol{\lambda}$ values that could result in DPP occupy a non-trivial space of $\mathbb{R}^d$. 

The reasoning of the prevalence of DPP is based on two assumptions: (1) $\boldsymbol{\mu}^p \neq \boldsymbol{\mu}^\pi$ and $\boldsymbol{\mu}^p \neq \boldsymbol{\mu}^L$; and (2) $\boldsymbol{\mu}^p-\boldsymbol{\mu}^\pi$ does not lie on the same line as $\boldsymbol{\mu}^p- \boldsymbol{\mu}^L$, thus $\mathcal{P}^\pi$ intersects $\mathcal{P}^L$ at an angle. The second assumption is equivalent to assuming that there does not exist non-zero constant $c$ such that $\boldsymbol{\mu}^p-\boldsymbol{\mu}^\pi = c(\boldsymbol{\mu}^p-\boldsymbol{\mu}^L)$. In other words, $\boldsymbol{\mu}^\pi\neq \boldsymbol{\mu}^L$ and $\boldsymbol{\mu}^p$ is not a weighted average of $\boldsymbol{\mu}^\pi$ and $\boldsymbol{\mu}^L$, with weights being $(w, 1-w)$, where $w= (1-c)^{-1}$. Some of the special cases that are ruled out from the assumptions above are as follows. (a) The case when $\boldsymbol{\mu}^p = \boldsymbol{\mu}^\pi$ or $\boldsymbol{\mu}^p =\boldsymbol{\mu}^L$ obviously violates the definition of DPP~\eqref{eqn:DPP_eqn}, thus does not result in DPP. (b) The case when $\boldsymbol{\mu}^{\pi} = \boldsymbol{\mu}^L$ results in DPP as long as $\boldsymbol{\mu}^p\neq \boldsymbol{\mu}^\pi$. This happens when the tail distributions of the prior and likelihood demonstrate certain properties, which is seen in numerical results of the Binomial example. (c) The case when $\boldsymbol{\mu}^p =w \boldsymbol{\mu}^\pi + (1-w)\boldsymbol{\mu}^L$ does not result in DPP as long as $w\neq 0$ or $1$, which is true since $c\neq 0$ or $\infty$. 


\tdplotsetmaincoords{60}{120}

\let\raarotold\raarot \let\rbarotold\rbarot
\let\rabrotold\rabrot \let\rbbrotold\rbbrot
\let\racrotold\racrot \let\rbcrotold\rbcrot

\let\raarot\racrotold \let\rbarot\rbcrotold
\let\rabrot\rabrotold \let\rbbrot\rbbrotold
\let\racrot\raarotold \let\rbcrot\rbarotold
\begin{figure}
    \centering
 \resizebox{6in}{4in}{ \begin{tikzpicture}[tdplot_main_coords]
  \fill[shade] (0,7.2,2.4)--(0,5.4,-2.7)--(0,0,0)--cycle;
  \fill[shade] (0,-3,-1)--(0,-2.2,1.1)--(0,0,0)--cycle;
  \draw[thick,->,gray] (0,0,0) -- (7,0,0) node[anchor=south]{$x$};
  \draw[thick,->,gray] (0,0,0) -- ( 0,7,0) node(Lpos)[anchor=west]{$y$};
  \draw[thick,-] (0,0,0) -- ( 0,-4,0) node(Lneg)[]{};
  \draw[thick,->,gray] (0,0,0) -- ( 0,0,7) node(pipos)[anchor=north east]{$z$};  
  \draw[thick,-] (0,0,0) -- ( 0,0,-7) node(pineg)[]{};  

  \pgfmathsetmacro{\ax}{5}
  \pgfmathsetmacro{\ay}{5}
  \pgfmathsetmacro{\az}{2}
  \draw[very thick,->,red] (0,0,0) -- (\ax,\ay,\az) node[anchor=west]{$\boldsymbol{\lambda}$};
  \draw[very thick,->,red, dashed] (0,0,0) -- (0,\ay,\az) node[anchor=north]{$\boldsymbol{\lambda}^{\rm proj}$};
  \draw[very thick,blue,->] (0,0,0) -- (0,3.9,6.5) node[anchor=west]{$\boldsymbol{\mu}^p-\boldsymbol{\mu}^L$};
  \draw[very thick,blue,->] (0,0,0) -- (0,-2,7) node[anchor=east]{$\boldsymbol{\mu}^p-\boldsymbol{\mu}^\pi$};
  \draw[ultra thick, blue,dashed] (0,7.2,2.4) -- (0,-3,-1) node[anchor=east]{$\mathcal{P}^{\pi}$};
  \draw[ultra thick, blue,dashed] (0,5.4,-2.7) -- (0,-2.2,1.1) node[anchor=east]{$\mathcal{P}^{L}$};
  \draw[dashed,gray,thick] (\ax,\ay,\az) -- (0,\ay,\az);
  \fill[blue,opacity=0.1] (0,0,13) -- (0,10,0) -- (0,0,-10) --  (0,-10,0)--cycle;
  \node at (0, -6, -1) {$\mathcal{S}^{\pi L}= {\rm Span}(\boldsymbol{\mu}^p-\boldsymbol{\mu}^L, \boldsymbol{\mu}^p-\boldsymbol{\mu}^\pi)$};
\end{tikzpicture}}
    \caption{Illustration of the geometry of DPP for general parameter dimensions. The light blue region, $\mathcal{S}^{\pi L}$, is the plane spanned by two linearly independent vectors $\boldsymbol{\mu}^p-\boldsymbol{\mu}^L$and $\boldsymbol{\mu}^p-\boldsymbol{\mu}^\pi$. The orthogonal spaces to $\boldsymbol{\mu}^p-\boldsymbol{\mu}^\pi$ and to $\boldsymbol{\mu}^p-\boldsymbol{\mu}^L$ within $\mathcal{S}^{\pi L}$ are respectively denoted by $\mathcal{P}^{\pi}$ and $\mathcal{P}^{L}$; thus $\mathcal{P}^{\pi}\perp \boldsymbol{\mu}^p-\boldsymbol{\mu}^\pi$ and  $\mathcal{P}^{L}\perp \boldsymbol{\mu}^p-\boldsymbol{\mu}^L$. The red vector is $\boldsymbol{\lambda}$ and the corresponding dashed vector is the projection of $\boldsymbol{\lambda}$ onto $\mathcal{S}^{\pi L}$, denoted by $\boldsymbol{\lambda}^{\rm proj}$. If $\boldsymbol{\lambda}^{\rm proj}$ lies in the two shaded triangular regions, the DPP does not occur; otherwise, the DPP occurs.}
    \label{fig:DPP_general}
\end{figure}

\section{DPP in Binomial Model Revisited}
\label{section:binomialmodel}

 The Binomial model employed by~\citet{xie2013incorporating} to analyze two-by-two contingency tables is covered by the LAN property in Section~\ref{section:conditions_DPP}, when the numbers of experimental trials go to infinity. We will not repeat the discussion for this case here, except to point out that in Appnedix~\ref{appendix:dpp_binomial}, we derive the DPP conditions for the Binomial model with Beta priors asymptotically. The condition corresponds well to the analytical forms in Example~\ref{proposition:2dimconstrast_heter_uncor_cov} in the exponential quadratic contrast case. This further demonstrates the generality of the exponential quadratic results, and offers the connection between the Beta-Binomial model with the Gaussian conjugate model.  
In practice, however, we care about the finite sample property of the inferential procedure, especially when the prior is moderately or highly informative. The case of finite numbers of trials does not fall into the realm of of exponential-quadratic likelihood. This section studies DPP in this finite sample scenario. 


Let $y_i \sim {\rm Binom}(n_i, p_i)$, $i=0, 1$, both $n_0$ and $n_1$ are finite. The parameter of interest is $\eta = p_1 - p_0$, for which we have ``some prior information''. Furthermore, we also have ``some prior information'' for $\alpha = p_0$. Using the notations in Section~\ref{sec:DPP_existence}, $\boldsymbol{\lambda}=(-1,1)^{\top}$ and $\boldsymbol{\theta} = (p_0, p_1)^{\top}$. This is an example given in~\citet{xie2013incorporating}. The likelihood is
\begin{align*}
L(p_0, \eta) &\propto p_0^{y_0} (1 - p_0)^{n_0- y_0} p_1^{y_1} (1 - p_1)^{n_1- y_1}\\
&= p_0^{y_0} (1 - p_0)^{n_0- y_0} (p_0+\eta)^{y_1} (1 - p_0-\eta)^{n_1- y_1}.
\end{align*}
The MLE of $p_0$ and $\eta$ are $\hat{p}_0 = \frac{y_0}{n_0}$ and $\hat{\eta} = \frac{y_1}{n_1} - \frac{y_0}{n_0}$. 
Let $\eta_0$ be the prior mean and $\eta^*$ be the posterior mode of $\eta$, then DPP occurs if and only if
\begin{equation}
\left[\eta^* - \left(\frac{y_1}{n_1} - \frac{y_0}{n_0}\right)\right] \left[\eta^* - \eta_0 \right] > 0. 
\label{eqn:dppbinomial}
\end{equation}
The (independent) conjugate prior for this model is given by $p_i \sim {\rm Beta}(a_i, b_i)$, $i = 0, 1$. This case is studied thoroughly in \citet{xie2013incorporating} thus we do not discuss this type of prior here. Instead, we focus on alternative priors and look at both theoretical and numerical results to offer insights on prior specification in non-Gaussian, non-linear models. Note that throughout this section, we consider the DPP under context that the posterior mode is the point estimate of the posterior, instead of the posterior mean; see definition~\ref{def:dpp}. 


\subsection{Theoretical Results}
\label{subsec:theory_binomial}

When both $n_0$ and $n_1$ are finite, the likelihood of the Binomial model is very different from exponential-quadratic type likelihoods thus we no longer have nice analytical solutions for the conditions of DPP as in Section~\ref{section:conditions_DPP}. To correspond as much as possible to the results obtained in Section~\ref{section:conditions_DPP}, we choose a truncated bivariate Gaussian prior for the parameters in the Binomial model and still try to express the posterior as a weighted average of prior mean and MLE. By doing so, we can examine the exact distinction (actually in terms of an extra residue term in the weighted average) of the Binomial model from the exponential-quadratic models. The truncation in the prior specification is adopted to account for constrained parameter space, i.e. $[0,1]$ for both $p_0$ and $p_1$, see e.g. \citet{balding2003likelihood} for an applied Bayesian analysis using truncated Gaussian prior in modeling probability parameter that lies in $[0,1]$. It is worth noting that we are working under the situation of known hyperparameters in the priors for $p_0$ and $p_1$, thus the normalization needed for the truncated Gaussian prior does not matter in the Bayesian inference (invariant to a known normalization constant).

Let the prior for $(\alpha=p_0, \eta=p_1-p_0)$ be truncated bivariate Gaussian with means $(\alpha_0, \eta_0)$ and variance-covariance matrix $\Sigma= \left( \begin{array}{cc}
\sigma_0^2 & r\sigma_1\sigma_0 \\
r\sigma_1\sigma_0 & \sigma_1^2
\end{array} \right)$. Assume that $\sigma_0>0, \sigma_1>0$ and $r\in[-1,1]$ are known constants. Thus the normalizing factor given by the truncation is a known constant. Proposition~\ref{proposition:postmode_partial_delta} rewrites the posterior mode as a weighted average of the prior mean and the MLE, plus an extra term, without which the DPP would not occur. 
\begin{proposition}[finite sample Binomial DPP]
For any fixed $r\in (-1, 1)$, the posterior mode $\eta^*$ satisfies
\begin{equation}
\eta = W_L \hat{\eta} + (1-W_L) \eta_0 +  W_{d} \left(\frac{y_0}{n_0} - \alpha_0\right),
\label{eqn:delta_partial_post_mode}
\end{equation}
where $I_0 = I(p_0, n_0)= \frac{n_0}{{p_0(1-p_0)}}$, $I_1 = I(p_1, n_1) = \frac{n_1}{p_1(1-p_1)}$, and
\begin{align*}
W_L &= W_L(p_0, p_1) = \frac{(1-r^2)I_0 I_1 + I_1[\frac{1}{\sigma_0^2} + \frac{r}{\sigma_0\sigma_1}]}{(1-r^2)I_0 I_1+ \frac{1}{\sigma_0^2\sigma_1^2 } + \frac{I_0}{\sigma_1^2}+ I_1 \left[\frac{1}{\sigma_0^2}+ 2\frac{r}{\sigma_0\sigma_1}+\frac{1}{\sigma_1^2}\right]}, \\
W_d &= W_d(p_0, p_1) = \frac{I_1\left(\frac{1}{\sigma_0^2} + \frac{r}{\sigma_0\sigma_1}\right)+ \frac{r}{\sigma_0\sigma_1} I_0 }{(1-r^2)I_0 I_1+ \frac{1}{\sigma_0^2\sigma_1^2} + \frac{I_0}{\sigma_1^2}+ I_1\left[\frac{1}{\sigma_0^2}+ 2\frac{r}{\sigma_0\sigma_1}+\frac{1}{\sigma_1^2}\right]}.
\end{align*}
Let ${y_0}/{n_0}\geq \alpha_0$ without loss of generality, then DPP occurs if and only if
\begin{align}
\eta_0 - \hat{\eta}  \in \left[ -  \frac{W_d}{1-W_L} \left(\frac{y_0}{n_0} - \alpha_0\right),  \frac{W_d}{W_L} \left(\frac{y_0}{n_0} - \alpha_0\right) \right]. \label{eqn:DPP_binomial_condition_gauss_prior}
\end{align}
Therefore, the severity of DPP depends on how large the interval on the right-hand-side is. 
\label{proposition:postmode_partial_delta}
\end{proposition}

When $r=0$, we can simplify the expressions in Proposition~\ref{proposition:postmode_partial_delta}. 
    \begin{align*}
        \frac{W_d}{W_L} &= \frac{1}{1+I_0 \sigma_0^2},\quad \frac{W_d}{1-W_L} = \frac{I_1\sigma_1^2}{1+(I_1+I_0)\sigma_0^2}.
    \end{align*}
Therefore, the larger $I_0\sigma_0^2$ is (or the smaller $I_1\sigma_1^2$ is), the shorter the interval on the right-hand-side of~\eqref{eqn:DPP_binomial_condition_gauss_prior} is, thus the less likely the DPP occurs. 
For any fixed $r$, when $\sigma_0,\sigma_1\rightarrow\infty$, $W_d\rightarrow 0$ and $W_L\rightarrow 1$; thus $\eta\rightarrow\hat{\eta}$. This corresponds to flat priors for $p_0, p_1$. These results match the phenomena we observe in numerical simulations, see Section~\ref{subsec:numerical_studies_binomial} for details.\\


From Section~\ref{section:conditions_DPP}, in the asymptotic sense, i.e. in the Gaussian model, when the 2-dimensional marginal contrast is the quantity of interest, the DPP does not occur if the marginal variances are homogeneous (Proposition~\ref{proposition:2dimconstrast_hom_cov} in Section~\ref{section:conditions_DPP}). But this is not possible here since the marginal variances for $p_1$ and $p_0$ are determined by $n_1,n_0$ and their respective values, which can be very different. As we show in the numerical examples in Section~\ref{subsection:numerical_results_gaussian}, in cases of heterogeneous marginal variances, the correlation structure has a significant impact on the probability that DPP occurs. For the Binomial model, the only freedom that we have is on the prior means and covariance matrices. We speculate that imposing correlations through the prior may not completely resolve the DPP but might alleviate the phenomenon, i.e., alter (hopefully reduce) the probability of occurrence. 
We next examine the impact of various correlations on the priors of transformed $(p_0, p_1)$. 
The proposition below considers bivariate Gaussian priors for the logit transformed parameters $p_0, p_1$. More studies based on numerical experiments are given in Section~\ref{subsec:numerical_studies_binomial}.
\begin{proposition}[transformed parameters in Binomial model]\label{proposition:binomial_logit}
Let $\xi_i = {\rm logit} (p_i)$, $i = 0, 1$. Assume that the prior for $\boldsymbol{\xi} = (\xi_0,\xi_1)$ is bivariate normal with mean $\boldsymbol{\mu}$ and variance-covariance $\Sigma = \left( \begin{array}{cc}
\sigma_0^2 & r\sigma_1\sigma_0 \\
r\sigma_1\sigma_0 & \sigma_1^2
\end{array} \right)$. Assume that $\sigma_1$ and $\sigma_2$ are known and that $r$ has a uniform prior on $(-1,1)$. Then at the posterior mode, $(p_1, r, p_0)$ satisfies
\begin{align*}
&\left(\frac{y_0}{n_0} - p_0\right)\left(\frac{y_1}{n_1} - p_1\right)  = -\frac{r}{n_0 n_1\sigma_0\sigma_1 (1-r^2)}.
\end{align*}
\end{proposition}
It shows that a positive correlation on the priors between logit-transformed $p_0, p_1$ incurs a negative correlation between the residuals, $\frac{y_0}{n_0} - p_0=\hat{p}_0-p_0$ and $\frac{y_1}{n_1} - p_1=\hat{p}_1-p_1$, on the posterior, and vice versa. Moreover, if we set $r = 0$, then at the posterior mode, either $p_0=y_0/n_0$ or $p_1 = y_1/n_1$, which is the MLE for $p_0$ or $p_1$. This is likely to result in DPP on $\eta = p_1 - p_0$, since the posterior mode already coincides with MLE. This heuristic argument gives partial evidence towards setting correlated priors for (transformed) $p_0$ and $p_1$ to alleviate DPP. We present numerical results on this situation to demonstrate our intuitions. 

\subsection{Numerical Results}
\label{subsec:numerical_studies_binomial}
Given the theoretical discussions in Section~\ref{subsec:theory_binomial}, we demonstrate the DPP for Binomial models via the numerical examples. 
Table~\ref{tab:binomial_example} summarizes the results under different prior specifications, corresponding to those given in Section~\ref{subsec:theory_binomial}. 

\begin{table}[tbph]
    \centering
    \begin{tabular}{c|cccc}
       Prior  & Posterior Mean & Posterior Median & Posterior Interval (95\%)  \\
       \hline
       Indep. Conj. & 0.237 & 0.240 & [0.094, 0.382] \\
       Gauss A $r = 0$ & 0.314 & 0.315 & [0.195, 0.427]\\
       Gauss A $r = 0.2$ &  0.325 & 0.325 & [0.213, 0.445]\\
       Gauss A, $r = -0.2$ & 0.308 & 0.310 & [0.189, 0.420] \\
       Gauss A, $r = 0.8$ & 0.381 &
0.381 & [0.296, 0.460] \\
       Gauss A, $r = -0.8$ & 0.250 & 0.247 & [0.153, 0.357]\\
       Gauss A, $r = 0.95$ &  0.403 & 0.402 & [0.345, 0.463]\\
       Gauss A, $r = -0.95$ & 0.200 & 0.203 & [0.121, 0.276] \\
       Gauss B, $r = 0$ & 0.094 & 0.097 & [-0.064, 0.248] \\
       Gauss B, $r = 0.2$ &  0.096 & 0.092 & [-0.061, 0.265]\\
       Gauss B, $r = -0.2$ & 0.092 & 0.093 & [-0.082, 0.249]\\
       Gauss B, $r = 0.8$ & 0.098 & 0.097 & [-0.059, 0.254] \\
       Gauss B, $r = -0.8$ & 0.099 & 0.102 & [-0.069, 0.250]\\
       Gauss B, $r = 0.95$ & 0.105 & 0.106 & [-0.035, 0.242] \\
       Gauss B, $r = -0.95$ & 0.119 &  0.120 & [-0.047, 0.282] \\
       \hline
    \end{tabular}
    \caption{Posteriors of $\eta=p_1-p_0$ in the Binomial model with independent conjugate priors with hyperparameters $(a, b)$ (row 1) and bivariate Gaussian priors (rows $2-15$) on $(p_0, p_1)$, where $r\in[-1,1]$ is the prior correlation between $p_0$ and $p_1$. The prior means are all given by $a/(a+b)$ ($2\times 1$ vector) and the marginal prior variances are given by $2\times 2$ diagonal matrices with diagonal elements given by $ab/((a+b)^2 (a+b+1))$ in ``Gauss A'' and $(0.5^2, 0.5^2)$ in ``Gauss B'', where $a = (14.66, 46.81), b = (4.88, 4.68)$. These hyperparameter choices are based on the example given in~\citet{xie2013incorporating}. The prior mean of $\eta$ is equal to $0.159$ and the MLE of $\eta$ is equal to $0.096$.}
    \label{tab:binomial_example}
\end{table}
As we can see from Table~\ref{tab:binomial_example}, having apriori negative correlations between $p_0$ and $p_1$ can alleviate the DPP though cannot diminish it, while the situation is worse when there is positive correlation between $p_0$ and $p_1$ a priori.  Having larger prior variances (Gauss B as opposed to Gauss A) can alleviate the DPP too: the prior impact becomes more and more negligible with a larger and larger prior variance. And in fact, in our case, the Gauss B variance is large enough to have eliminated the DPP in several cases (ones with large absolute correlations). Furthermore, we note that we also list the posterior intervals in Table~\ref{tab:binomial_example}, just to illustrate that the posterior interval covers the prior mean and MLE in some cases while misses one or both in other cases. Again, we do not discuss alternative definitions of DPP that extends the framework beyond point estimates in this paper. However, the lengths of the posterior intervals, together with the posterior mean values, reveal how bad/moderate the DPP is in each different case.

Next we examine some other more ``flexible'' priors (with either only unknown prior parameter $r$ or both unknown $\boldsymbol\sigma$ and $r$) on this Binomial example.  We can see from Table~\ref{tab:binomial_continued_otherpriors} that if we do not fix the prior variances and fit the variance parameters, the DPP does not occur. This is because the data is providing information to the prior specification thus the prior and likelihood are more aligned. And DPP is avoided.

\begin{table}[tbph]
    \centering
    \begin{tabular}{c|cccc}
       Transformation & Prior Var. & Post. Mean & Post. Interval (95\%) & Est. $r$ (95\% Post.)\\
       \hline
       None & A, Unknown $r$ & 0.388 & [0.297, 0.458] & 0.872 ([0.452, 0.992]) \\ 
       None & B, Unknown $r$ & 0.103 & [-0.065, 0.250] & 0.458 ([-0.552, 0.992])\\ 
None & Unknown $(\boldsymbol{\sigma},r)$ & 0.096& [-0.068, 0.274] & 0.254 ([-0.856, 0.990]) \\
Logit & Unknown $(\boldsymbol{\sigma},r)$ & 0.117 &[-0.032, 0.255] & 0.688 ([-0.036, 0.984])\\
\hline
    \end{tabular}
    \caption{The same Binomial model as in Table~\ref{tab:binomial_example} with different priors. The prior mean is equal to $0.159$ and MLE is equal to $0.096$. ``A'' represents when $\boldsymbol{\sigma}^2 = ab/((a+b)^2(a+b+1))$ and ``B'' represents when $\boldsymbol{\sigma}^2 = (0.5^2, 0.5^2)$. When $\boldsymbol{\sigma}$ is unknown, the prior is given by independent Gamma with hyperparameters $(10, 10)$. When the covariance matrix is unknown (both $\boldsymbol{\sigma}$ and $\rho$ unknown), the covariance matrix is given flat prior.}
    \label{tab:binomial_continued_otherpriors}
\end{table}

In summary, by studying the DPP for the Binomial model both theoretically and numerically, we confirm our conjecture that the guidance given by studying the DPP for exponential-quadratic likelihoods can also be applied when the likelihood is far from being exponential-quadratic. The DPP for general likelihood cases are slightly more tricky than exponential-quadratic cases, as we demonstrate in this section. Thus practitioners working with highly non-exponential-quadratic likelihoods shall be more cautious and be aware of the pitfalls of using informative priors.

\section{Conclusions and Discussions}
\label{section:summary}

In this paper, we derive conditions for DPP under exponential quadratic likelihoods and demonstrate DPP using numerical experiments for Gaussian and Binomial models. The investigations on DPP are of interest in applications and have practical implications on the choice of priors, especially informative priors, and normalization (pre-processing) of data. 

From studying the DPP under the exponential-quadratic likelihood, we recommend the following practical guidelines to help set up models to avoid or mitigate DPP for Bayesian analysis: {(I)}
having uncorrelated dimensions for the parameters, in both the prior and likelihood, is desired to alleviate or avoid DPP; 
{(II)}
re-scale or re-parameterize such that the prior is not super skewed, and homogeneous variance across dimensions is desired; 
{(III)} transform data such that the likelihood is not highly skewed and, similarly, homogeneous variance across dimensions is desired; {(IV)} 
in case the different dimensions of the parameter are correlated in both the prior and the likelihood, making sure that the correlation patterns are the same between prior and likelihood could alleviate or prevent DPP. With a digression from the exponential-quadratic likelihood, we also show through the Binomial model that the suggestions above shall still help alleviate DPP in highly non-linear non-Gaussian cases. Furthermore, setting hyper-priors could be helpful when DPP is suspected: the data would inform the estimation of the hyperparameters in the prior such that the DPP is mitigated while not fully avoided. This is reflected from numerical studies of the Binomial models. 

Several of these guidelines are consistent with a number of current practices in Bayesian inference to ease computational burden. 
Example~\ref{proposition:2dimconstrast_hom_cov} shows that, when possible, normalization of the data, together with proper re-scaling of the parameters prior to analysis to make the marginal variances of the parameters close to being homogeneous, is a recommended step for avoiding/mitigating the DPP. In practice, for most computational algorithms~\citep{robert2013monte}, such as the Markov chain Monte Carlo (MCMC, see~\citet{liu2008monte} and references therein), it is also easier to tune if the different parameters lie on similar scales, such as being close to standard Gaussian distribution in the model. A concrete example is the Neal's Funnel~\citep{papaspiliopoulos2007general} in the Hamiltonian Monte Carlo~\citep{neal2011mcmc} implementation in the Stan package~\citep{hoffman2014no}, it is demonstrated that ``reparameterization can dramatically increase effective sample size for the same number of iterations or even make programs that would not converge well behaved''; see the reparameterization section in the \texttt{Stan User's Guide}~\citep{stan2018}. Thus the geometry of the prior-likelihood alignment impacts both the behavior of the Bayes estimator and the performance of computational algorithms for posterior sampling. In this paper, from the perspective of avoiding potential DPP issue, we are reassuring the importance of these guidance for practitioners of Bayesian inference, especially under informative priors. 

Finally, we would like to point out that although this paper focused on the DPP for point estimation only, the phenomenon extends to set and distributional inference as well. \citet{xie2013incorporating} demonstrated the DPP with the Binomial example using credible and confidence intervals at different levels. We discuss the DPP in point estimation, focusing on the relationship between the posterior mean, the prior mean and the MLE,  to provide a simplified and essential insight into the phenomenon. Extended investigation of the DPP in set and distributional inference and their respective practical implication are left to future work.



\bibliographystyle{imsart-nameyear}
\bibliography{paper-ref}

\begin{appendix}

\section{Proof of Theorem~\ref{theorem:conditionsDPP_gauss}}\label{appendix:conditionsDPP_gauss}
The posterior distribution for $\boldsymbol{\theta}$ is 
\begin{equation*}
\mathcal{N}\left(\boldsymbol{\theta}; \Sigma^{p}\left[\left(\Sigma^{\pi}\right)^{-1}\boldsymbol{\mu}^{\pi} + \left(\Sigma^{L}\right)^{-1}\boldsymbol{\mu}^L\right], \Sigma^{p} = \left[(\Sigma^{\pi})^{-1}+(\Sigma^L)^{-1}\right]^{-1}\right).
\end{equation*}
The posterior distribution for $\eta=\boldsymbol{\lambda}^\top\boldsymbol{\theta}$ is then
\begin{equation*}
\mathcal{N}\left(\boldsymbol{\theta}; \boldsymbol{\lambda}^{\top}\Sigma^{p}\left(\Sigma^{\pi}\right)^{-1}\boldsymbol{\mu}^{\pi} + \boldsymbol{\lambda}^{\top}\Sigma^{p}\left(\Sigma^{L}\right)^{-1}\boldsymbol{\mu}^L, \boldsymbol{\lambda}^{\top}\Sigma^{p}\boldsymbol{\lambda}\right).
\end{equation*}
The discrepant posterior phenomenon (DPP) does not occur if and only if 
$$\left(\mathbb{E}_{\rm post}(\boldsymbol{\lambda}^{\top}\boldsymbol{\theta}) - \boldsymbol{\lambda}^{\top} \boldsymbol{\mu}^{\pi}\right)\left(\mathbb{E}_{\rm post}(\boldsymbol{\lambda}^{\top}\boldsymbol{\theta}) - \boldsymbol{\lambda}^{\top} \boldsymbol{\mu}^{L}\right)\leq 0,$$
in other words, DPP occurs if 
\begin{equation*}
\mathbb{E}_{\rm post}(\boldsymbol{\lambda}^{\top}\boldsymbol{\theta}) :=\boldsymbol{\lambda}^{\top}\Sigma^{p}\left(\Sigma^{\pi}\right)^{-1}\boldsymbol{\mu}^{\pi} + \boldsymbol{\lambda}^{\top}\Sigma^{p}\left(\Sigma^{L}\right)^{-1}\boldsymbol{\mu}^L \quad
\left\{
\begin{array}{c}
> \max\left\{ \boldsymbol{\lambda}^{\top} \boldsymbol{\mu}^{\pi}, \boldsymbol{\lambda}^{\top} \boldsymbol{\mu}^{L} \right\}\vspace{6pt}\\
\text{ or } < \min\left\{ \boldsymbol{\lambda}^{\top} \boldsymbol{\mu}^{\pi}, \boldsymbol{\lambda}^{\top} \boldsymbol{\mu}^{L} \right\}
\end{array}
\right.
\end{equation*}

\section{Proof of Theorem~\ref{cor1}}
Since both $\boldsymbol{\lambda}^{\top}\Sigma^{p}\Lambda^{-1}\left(\overline{\bf y}_n - \boldsymbol{\mu}^{\pi}\right)$ and $\left(\overline{\bf y}_n - \boldsymbol{\mu}^{\pi}\right)^{\top}\left(\Sigma^{\pi}\right)^{-1} \Sigma^{p}\boldsymbol{\lambda}$ follows univariate Gaussian distributions under the data generating model, the probability that DPP occurs is equal to zero if and only if they are perfectly positively correlated, i.e. there exists a positive constant $c$ such that $\boldsymbol{\lambda}^{\top}\Sigma^{p}(\Lambda^{-1} - c \left(\Sigma^{\pi}\right)^{-1} )\left(\overline{\bf y}_n - \boldsymbol{\mu}^{\pi}\right) = 0$ holds with probability $1$. This implies that $\boldsymbol{\lambda}^{\top}\Sigma^{p}(\Lambda^{-1} - c \left(\Sigma^{\pi}\right)^{-1})=0$. 

\section{Proof of Theorem~\ref{theorem:gaussian_existence}}
\label{appendix:gaussian_existence_proof}

The proof of the theorem is straightforward, and we briefly explain it here. If $\boldsymbol{\mu}^{\pi} = \boldsymbol{\mu}^L$, then $\boldsymbol{\mu}^p = \boldsymbol{\mu}^\pi$ thus the DPP does not appear. From Section~\ref{appendix:geometry_dpp_general}, if $\boldsymbol{\mu}^p-\boldsymbol{\mu}^\pi$ and $\boldsymbol{\mu}^p-\boldsymbol{\mu}^L$  does not lie on the same line, there exists a nontrivial space for possible $\boldsymbol{\lambda}$ such that the DPP occurs. In the Gaussian setting, $\boldsymbol{\mu}^p= \Sigma^{p}\left[\left(\Sigma^{\pi}\right)^{-1}\boldsymbol{\mu}^{\pi} + \left(\Lambda/n\right)^{-1}\boldsymbol{\mu}^L\right]$, where $\left({\Sigma^{p}}\right)^{-1} = \left(\Sigma^{\pi}\right)^{-1} + \left(\Lambda/n\right)^{-1}$. Note that the vectors  $\boldsymbol{\mu}^p -\boldsymbol{\mu}^\pi$ and $\boldsymbol{\mu}^p-\boldsymbol{\mu}^L$ lying on the same line is equivalent to saying that there exist some constant $c\neq 0$, such that $(1-c)\boldsymbol{\mu}^p =  \boldsymbol{\mu}^\pi -c\boldsymbol{\mu}^L$. This is equivalent to the following linear equation: 
\begin{align}\label{eq:thm1}
 \left( c\left(\Sigma^{\pi}\right)^{-1} + \left(\Lambda/n\right)^{-1}\right) (\boldsymbol{\mu}^L - \boldsymbol{\mu}^\pi)=0.
\end{align}
This is a linear equation for sample mean $\boldsymbol{\mu}^L = \overline{\bf y}_n$. Thus, for a given $\boldsymbol{\mu}^\pi$ and as long as $\left( c\left(\Sigma^{\pi}\right)^{-1} + \left(\Lambda/n\right)^{-1}\right) \not = 0$, the probability that equation (\ref{eq:thm1}) holds is zero. Thus the DPP prevalence holds with probability $1$ as long as $\Lambda$ is not a multiplier of $\Sigma^\pi$. 

\section{Proofs of Examples~\ref{proposotion:generaldim_Gaussian_diagonal_cov} and~\ref{proposotion:generaldim_Gaussian_homo_cor}}

\begin{proof}
\textbf{\ref{proposotion:generaldim_Gaussian_diagonal_cov}.} This is straightforward from Theorem~\ref{theorem:conditionsDPP_gauss} thus detailed proof is omitted.\\
\textbf{\ref{proposotion:generaldim_Gaussian_homo_cor}.} From Sherman-Morrison formula, we have 
\begin{equation*}
(\Sigma^{\pi})^{-1} = \frac{\sigma_{\pi}^{-2}}{1-r} \left[I_d - \frac{r}{rd + 1 - r} \boldsymbol{1}_d \boldsymbol{1}_d^{\top}\right]\quad \text{and} \quad(\Sigma^{L})^{-1} = \frac{\sigma_{L}^{-2}}{1-\rho} \left[I_d - \frac{\rho}{\rho d + 1 - \rho} \boldsymbol{1}_d \boldsymbol{1}_d^{\top}\right].
\end{equation*}
Therefore, 
\begin{align*}
\Sigma^{p} (\Sigma^L)^{-1} 
&= W_{r\rho} I_d - C_{r\rho}\boldsymbol{1}_d \boldsymbol{1}_d^{\top},\\
\Sigma^{p} (\Sigma^\pi)^{-1} &= (1-W_{r\rho}) I_d + C_{r\rho}\boldsymbol{1}_d \boldsymbol{1}_d^{\top}.
\end{align*}
Thus $\Delta_1
= W_{r\rho} d_{L\pi}^{(1)} - C_{r\rho} d_{L\pi}^{(2)}$ and $\Delta_2
= (1-W_{r\rho}) d_{L\pi}^{(1)} + C_{r\rho} d_{L\pi}^{(2)}$. Therefore, 
\begin{align*}
\Delta_1 \Delta_2 &= W_{r\rho}(1-W_{r\rho})\left[d_{L\pi}^{(1)}\right]^2 - C_{r\rho}^2\left[d_{L\pi}^{(2)}\right]^2 + C_{r\rho} (2 W_{r\rho} - 1)d_{L\pi}^{(1)}d_{L\pi}^{(2)}.
\end{align*}
The rest follows directly from Theorem~\ref{theorem:conditionsDPP_gauss}. 
\end{proof}

\section{Statement and Proof of Corollary to Proposition~\ref{proposition:2dimconstrast_hom_cov}}
\label{appendix:statement_proof_gauss_homo_corollary}

\begin{corollary} 
Let $\Sigma^{\pi} = \sigma_{\pi}^2 {\rm diag}(1, 1)$, $\Sigma^{L} = \sigma_{L}^2\left(\begin{array}{cc}
 1 & \rho\\ \rho & 1
 \end{array}\right)$. Then (1) the posterior distribution for $\theta_1 - \theta_2$ is $\mathcal{N}(\Delta^*, 2(1-\rho)\left[(1-\rho)\sigma_{\pi}^{-2} + \sigma_L^{-2}\right]^{-1})$, where $\Delta^* = w_{\pi} \Delta_{\pi} + w_L \Delta_L$, $w_{\pi}= \frac{(1-\rho)\sigma_{\pi}^{-2}}{(1-\rho)\sigma_{\pi}^{-2} +  \sigma_{L}^{-2}}$, and $w_L = 1-w_{\pi}$; and (2) $\min\{\Delta_{\pi}, \Delta_{L}\}\leq \Delta^*\leq \max\{\Delta_{\pi}, \Delta_{L}\}$, i.e. no DPP. Same results hold when $\Sigma^{L} = \sigma_{L}^2 {\rm diag}(1, 1)$, $\Sigma^{\pi} = \sigma_{\pi}^2\left(\begin{array}{cc}
 1 & \rho\\ \rho & 1
 \end{array}\right)$ due to the symmetry of likelihood and prior. 
\label{corollary:gaussian_homo}
\end{corollary}
\begin{proof}
We use the same notations as in Lemmas~\ref{lemma:bivariateNormaldiff} and~\ref{lemma:bivariatenormal_plp}. In this case, $$\left(\Sigma^\pi\right)^{-1} = \sigma_{\pi}^{-2} \left(\begin{array}{cc}
 1 & 0\\ 0 & 1
 \end{array}\right)\quad \text{and}\quad \left(\Sigma^L\right)^{-1} = \frac{\sigma_{L}^{-2}}{1-\rho^2}\left(\begin{array}{cc}
 1 & -\rho\\ -\rho & 1
 \end{array}\right).$$
 Consequently, we have
 \begin{align*}
  \Sigma^{p}  &= \left[\sigma_{\pi}^{-2} + \frac{\sigma_L^{-2}}{1-\rho^2}\right]^{-1} \frac{1}{1-r^2} \left(\begin{array}{cc}
 1 & -r\\ -r & 1
 \end{array}\right), r = \frac{-\frac{\rho}{1-\rho^2}\sigma_L^{-2}}{\sigma_{\pi}^{-2}+\frac{\sigma_L^{-2}}{1-\rho^2}}.\\
 T\Sigma^{p} T^{\top} &=\left[\sigma_{\pi}^{-2} + \frac{\sigma_L^{-2}}{1-\rho^2}\right]^{-1} \frac{1}{1-r^2} \left(\begin{array}{cc}
 2+2r & 0\\ 0 & 2-2r
 \end{array}\right). \\
 \left(T\Sigma^{p} T^{\top}\right)_{11} &=\left[\sigma_{\pi}^{-2} + \frac{\sigma_L^{-2}}{1-\rho^2}\right]^{-1} \frac{2}{1-r} = 2\left[\sigma_{\pi}^{-2} + \frac{\sigma_L^{-2}}{1-\rho}\right]^{-1}. 
 \end{align*}
 Furthermore, we have
 \begin{align*}
 T \Sigma^{p} \left(\Sigma^{\pi}\right)^{-1} \boldsymbol{\mu}^{\pi} & 
 = \left(1+\frac{\sigma_\pi^2\sigma_{L}^{-2}}{1-\rho^2}\right)^{-1} \left(\begin{array}{c}
 \frac{\mu_1^\pi-\mu_2^\pi}{1-\tilde{r}}\\
 \frac{\mu_1^\pi+\mu_2^\pi}{1+\tilde{r}}
 \end{array}\right),
 \tilde{r} = -\left(1+\frac{\sigma_\pi^2\sigma_{L}^{-2}}{1-\rho^2}\right)^{-1} \frac{\rho\sigma_\pi^2\sigma_{L}^{-2}}{1-\rho^2};\\
 T\Sigma^{p}\left(\Sigma^L\right)^{-1} \boldsymbol{\mu}^L &= \left(1+\sigma_L^2\sigma_{\pi}^{-2}\right)^{-1} \left(\begin{array}{c}
 \frac{\mu_1^L-\mu_2^L}{1-\tilde{r}^*}\\
 \frac{\mu_1^L+\mu_2^L}{1+\tilde{r}^*}
 \end{array}\right), \tilde{r}^* = \frac{\rho \sigma_L^2\sigma_{\pi}^{-2}}{1+\sigma_L^2\sigma_{\pi}^{-2}}.
 \end{align*}
 Using notations in Lemma~\ref{lemma:bivariatenormal_plp}, we have $\mu_1^*-\mu_2^* = w_{\pi} \Delta_{\pi} + w_L \Delta_L$, where
 \begin{align*}
  \mu_1^*-\mu_2^* 
 &= w_{\pi} \Delta_{\pi} + w_L \Delta_L,
 w_{\pi}= \frac{(1-\rho)\sigma_{\pi}^{-2}}{(1-\rho)\sigma_{\pi}^{-2} +  \sigma_{L}^{-2}},  w_L = 1-w_{\pi}.
\end{align*}
 Therefore, $\min\{\Delta_{\pi}, \Delta_{L}\}\leq\mu_1^*-\mu_2^*\leq \max\{\Delta_{\pi}, \Delta_{L}\}$ if and only if
 \begin{equation*}
 \left([1-w_\pi]\Delta_{\pi}-w_L\Delta_L\right)\left([1-w_L]\Delta_{L}-w_{\pi}\Delta_{\pi}\right) \leq 0,
 \end{equation*}
 which is always true since $0\leq w_{\pi}\leq 1$. 
\end{proof}

\section{Proof of Proposition~\ref{proposition:2dimconstrast_hom_cov}}
\begin{proof}
We use the same notations as in Lemmas~\ref{lemma:bivariateNormaldiff} and~\ref{lemma:bivariatenormal_plp}. In this case, $$\left(\Sigma^\pi\right)^{-1} = \frac{\sigma_{\pi}^{-2}}{1-r^2} \left(\begin{array}{cc}
1 & -r\\ -r & 1
\end{array}\right)\quad \text{and}\quad \left(\Sigma^L\right)^{-1} = \frac{\sigma_{L}^{-2}}{1-\rho^2}\left(\begin{array}{cc}
1 & -\rho\\ -\rho & 1
\end{array}\right).$$
Consequently, we have
\begin{align*}
 \Sigma^{p}  &= \left[\frac{\sigma_{\pi}^{-2}}{1-r^2} + \frac{\sigma_L^{-2}}{1-\rho^2}\right]^{-1} \frac{1}{1-s^2} \left(\begin{array}{cc}
1 & -s\\ -s & 1
\end{array}\right), s = \frac{-\frac{\rho}{1-\rho^2}\sigma_L^{-2} -\frac{r}{1-r^2}\sigma_{\pi}^{-2}}{\frac{\sigma_{\pi}^{-2}}{1-r^2}+\frac{\sigma_L^{-2}}{1-\rho^2}}.\\
T\Sigma^{p} T^{\top} &=\left[\frac{\sigma_{\pi}^{-2}}{1-r^2} + \frac{\sigma_L^{-2}}{1-\rho^2}\right]^{-1} \frac{1}{1-s^2} \left(\begin{array}{cc}
2+2s & 0\\ 0 & 2-2s
\end{array}\right). \\
\left(T\Sigma^{p} T^{\top}\right)_{11} &=\left[\frac{\sigma_{\pi}^{-2}}{1-r^2} + \frac{\sigma_L^{-2}}{1-\rho^2}\right]^{-1} \frac{2}{1-s} = 2\left[\frac{\sigma_{\pi}^{-2}}{1-r} + \frac{\sigma_L^{-2}}{1-\rho}\right]^{-1}. 
\end{align*}
Furthermore, we have
\begin{align*}
T \Sigma^{p} \left(\Sigma^{\pi}\right)^{-1} \boldsymbol{\mu}^{\pi} & 
= \left(\frac{1}{1-r^2}+\frac{\sigma_\pi^2\sigma_{L}^{-2}}{1-\rho^2}\right)^{-1} \left(\begin{array}{c}
\frac{\mu_1^\pi-\mu_2^\pi}{(1-r)(1-s)}\\
\frac{\mu_1^\pi+\mu_2^\pi}{(1+r)(1+s)}
\end{array}\right);\\
T\Sigma^{p}\left(\Sigma^L\right)^{-1} \boldsymbol{\mu}^L &= \left(\frac{1}{1-\rho^2}+\frac{\sigma_L^2\sigma_{\pi}^{-2}}{1-r^2}\right)^{-1} \left(\begin{array}{c}
\frac{\mu_1^L-\mu_2^L}{(1-s)(1-\rho)}\\
\frac{\mu_1^L+\mu_2^L}{(1+s)(1+\rho)}
\end{array}\right).
\end{align*}
Therefore, from Lemma~\ref{lemma:bivariatenormal_plp}, the posterior distribution for $\theta_1 - \theta_2$ is $\mathcal{N}(\Delta^*, 2 \left[\frac{\sigma_{\pi}^{-2}}{1-r} + \frac{\sigma_L^{-2}}{1-\rho}\right]^{-1})$, where $\Delta^* = w_{\pi} \Delta_{\pi} + w_L \Delta_L$, 
\begin{align*}
w_{\pi}&= \frac{(1-\rho)\sigma_{\pi}^{-2}}{(1-\rho)\sigma_{\pi}^{-2} + (1-r) \sigma_{L}^{-2}}\quad \text{and}\quad w_L = 1-w_{\pi}.
\end{align*}
Therefore, $\min\{\Delta_{\pi}, \Delta_{L}\}\leq\Delta^*\leq \max\{\Delta_{\pi}, \Delta_{L}\}$ if and only if
\begin{equation*}
\left([1-w_\pi]\Delta_{\pi}-w_L\Delta_L\right)\left([1-w_L]\Delta_{L}-w_{\pi}\Delta_{\pi}\right) \leq 0,
\end{equation*}
which is always true since $0\leq w_{\pi}\leq 1$. Thus there is no DPP. 
\end{proof}

\section{Proof of Proposition~\ref{proposition:2dimconstrast_heter_uncor_cov}}
\begin{proof}
We use the same notations as in Lemmas~\ref{lemma:bivariateNormaldiff} and~\ref{lemma:bivariatenormal_plp}.
Since $\left(\Sigma^\pi\right)^{-1} =  \left(\begin{array}{cc}
s_{\pi}^{-2}  & 0\\ 0 & \sigma_{\pi}^{-2} 
\end{array}\right)$ and $\left(\Sigma^L\right)^{-1} = \left(\begin{array}{cc}
s_L^{-2} & 0\\ 0 & \sigma_L^{-2}
\end{array}\right)$, we have $
\Sigma^{p} = \left( \begin{array}{cc}
\left[s_{\pi}^{-2}+ s_L^{-2}\right]^{-1} & 0\\
 0 & \left[\sigma_{\pi}^{-2}+ \sigma_L^{-2}\right]^{-1}\end{array}\right)$. 
Therefore, $\left(T\Sigma^{p} T^{\top}\right)_{11} = \left[s_{\pi}^{-2}+ s_L^{-2}\right]^{-1} + \left[\sigma_{\pi}^{-2}+ \sigma_L^{-2}\right]^{-1}$. Furthermore, 
\begin{align*}
\left[T \Sigma^{p} \left(\Sigma^{\pi}\right)^{-1} \boldsymbol{\mu}^{\pi} \right]_1& 
= \frac{s_{\pi}^{-2}}{s_{\pi}^{-2} + s_{L}^{-2}} \mu_1^{\pi} - \frac{\sigma_{\pi}^{-2}}{\sigma_{\pi}^{-2} + \sigma_{L}^{-2}}\mu_2^{\pi};\\
\left[T\Sigma^{p}\left(\Sigma^L\right)^{-1} \boldsymbol{\mu}^L\right]_1 &= \frac{s_{L}^{-2}}{s_{\pi}^{-2} + s_{L}^{-2}} \mu_1^{L} - \frac{\sigma_{L}^{-2}}{\sigma_{\pi}^{-2} + \sigma_{L}^{-2}}\mu_2^{L}.
\end{align*}
Following Lemma~\ref{lemma:bivariatenormal_plp}, we have $\Delta^* = w_{\pi}^s \mu_1^{\pi} - w^{\sigma}\mu_2^{\pi} + (1-w_{\pi}^s) \mu_1^{L} - (1-w^{\sigma})\mu_2^{L}$,
where $w^{s} = \frac{s_{\pi}^{-2}}{s_{\pi}^{-2} + s_{L}^{-2}}$ and $w^{\sigma} = \frac{\sigma_{\pi}^{-2}}{\sigma_{\pi}^{-2} + \sigma_{L}^{-2}}$. The posterior distribution for $\theta_1 - \theta_2$ is
\begin{equation*}
\mathcal{N}\left(\Delta^*,\left[s_{\pi}^{-2}+ s_L^{-2}\right]^{-1} + \left[\sigma_{\pi}^{-2}+ \sigma_L^{-2}\right]^{-1} \right). 
\end{equation*}
Therefore, $\min\{\Delta_{\pi}, \Delta_{L}\}\leq\Delta^*\leq \max\{\Delta_{\pi}, \Delta_{L}\}$ if and only if $(\Delta^* - \Delta_{\pi})(\Delta^* - \Delta_{L})\leq 0$, i.e.
\begin{align*}
\left[w_{\pi}^s \mu_1^{\pi} - w^{\sigma}\mu_2^{\pi} + (1-w_{\pi}^s) \mu_1^{L} - (1-w^{\sigma})\mu_2^{L} - (\mu_1^{\pi} - \mu_2^{\pi})\right]\times\\
\left[w_{\pi}^s \mu_1^{\pi} - w^{\sigma}\mu_2^{\pi} + (1-w_{\pi}^s) \mu_1^{L} - (1-w^{\sigma})\mu_2^{L} - (\mu_1^{L} - \mu_2^{L})\right] \leq 0. \end{align*}
This is equivalent to, if we denote $\delta_1 = \mu_1^L - \mu_1^{\pi}$ and $\delta_2 = \mu_2^L - \mu_2^{\pi}$,
\begin{equation*}
w_{\pi}^s(1-w_{\pi}^s)  \left[ \delta_1 - \frac{1-w^{\sigma}}{1-w_{\pi}^s}\delta_2 \right]\times \left[\delta_1 - \frac{w^{\sigma}}{w_{\pi}^s } \delta_2 \right] \geq 0. 
\end{equation*}
In other words, $\min\{\Delta_{\pi}, \Delta_{L}\}\leq\Delta^*\leq \max\{\Delta_{\pi}, \Delta_{L}\}$ if and only if
\begin{align*}
\delta_1 \geq \max\left\{\frac{1-w^{\sigma}}{1-w_{\pi}^s} \delta_2,  \frac{w^{\sigma}}{w_{\pi}^s } \delta_2\right\}\quad \text{or} \quad \delta_1 \leq \min\left\{\frac{1-w^{\sigma}}{1-w_{\pi}^s} \delta_2,  \frac{w^{\sigma}}{w_{\pi}^s } \delta_2\right\}. 
\end{align*}
This holds if and only if $\mu_2^L= \mu_2^{\pi}$ or
\begin{align*}
\frac{\mu_1^L - \mu_1^{\pi}}{\mu_2^L - \mu_2^{\pi}} \geq \max\left\{\frac{1-w^{\sigma}}{1-w_{\pi}^s} ,  \frac{w^{\sigma}}{w_{\pi}^s }  \right\}, \quad \text{or} \quad \frac{\mu_1^L - \mu_1^{\pi}}{\mu_2^L - \mu_2^{\pi}} \leq \min\left\{\frac{1-w^{\sigma}}{1-w_{\pi}^s}  ,  \frac{w^{\sigma}}{w_{\pi}^s }  \right\}. 
\end{align*}
In the special case when $w^{\sigma}=w^{s}$, DPP does not occur since either inequality above holds.
\end{proof}

\section{Lemmas}

\begin{lemma}
Let $\boldsymbol{\theta}\sim \mathcal{N}(\boldsymbol{\mu},\Sigma)$, where $\boldsymbol{\mu} = (\mu_1,\mu_2)^{\top}$, $\Sigma = \left(\begin{array}{cc}
\Sigma_{11} & \Sigma_{12}\\
\Sigma_{21} & \Sigma_{22}
\end{array}\right)$. Then
\begin{equation*}
\theta_1 - \theta_2 \sim \mathcal{N}(\mu_1 - \mu_2, \Sigma_{11} - \Sigma_{12} - \Sigma_{21} + \Sigma_{22}).
\end{equation*}
\label{lemma:bivariateNormaldiff}
\end{lemma}

\begin{proof}
Let $\tilde{\boldsymbol{\theta}} = (\tilde{\theta}_1, \tilde{\theta}_2)^{\top} = T\boldsymbol{\theta}$, where $T = \left(\begin{array}{cc}
1 & -1\\ 1 & 1
\end{array}\right)$. Then $\tilde{\theta}_1 = \theta_1 - \theta_2$ and
\begin{equation*}
\tilde{\boldsymbol{\theta}}\sim \mathcal{N}\left( T\boldsymbol{\mu}, T\Sigma T^{\top}\right) =  \mathcal{N}\left( \left(\begin{array}{c} 
\mu_1 - \mu_2\\ \mu_1+\mu_2
\end{array}\right), \left(\begin{array}{cc}
\Sigma_{11}-\Sigma_{21}-\Sigma_{12}+\Sigma_{22} & \Sigma_{11}-\Sigma_{22}\\
\Sigma_{11}-\Sigma_{22} & \Sigma_{11}+\Sigma_{21}+\Sigma_{12}+\Sigma_{22} 
\end{array}\right) \right).
\end{equation*}
\end{proof}

\begin{lemma}
Let the prior for a 2-dimensional parameter $\boldsymbol{\theta}$ be $\mathcal{N}(\boldsymbol{\mu}^{\pi}, \Sigma^{\pi})$ and the likelihood be proportional to $\phi(\boldsymbol{\theta}; \boldsymbol{\mu}^{L}, \Sigma^{L})$ where $\phi$ denotes Gaussian density. The prior distribution for $\theta_1 - \theta_2$ is
\begin{equation*}
 \mathcal{N}(\mu_1^{\pi} - \mu_2^{\pi}, \Sigma_{11}^{\pi} - \Sigma_{12}^{\pi} - \Sigma_{21}^{\pi} + \Sigma_{22}^{\pi}).
\end{equation*}
The marginal likelihood for $\theta_1-\theta_2$ is proportional to  
\begin{equation*}
\phi(\theta_1 - \theta_2; \mu_1^{L} - \mu_2^{L}, \Sigma_{11}^{L} - \Sigma_{12}^{L} - \Sigma_{21}^{L} + \Sigma_{22}^{L}).
\end{equation*}
The posterior distribution for $\boldsymbol{\theta}$ is $\mathcal{N} (\boldsymbol{\mu}^p, \Sigma^{p})$ where
\begin{equation*}
\Sigma^{p} =  \left[\left(\Sigma^{\pi}\right)^{-1} + \left(\Sigma^{L}\right)^{-1}\right]^{-1} \quad {\text and}\quad \boldsymbol{\mu}^p = \Sigma^{p} \left( \left(\Sigma^{\pi}\right)^{-1} \boldsymbol{\mu}^{\pi} + \left(\Sigma^{L}\right)^{-1} \boldsymbol{\mu}^{L} \right). 
\end{equation*}
The posterior distribution for $\theta_1-\theta_2$ is $\mathcal{N}({\mu}_1^{*} - \mu_2^{*}, (T\Sigma^{p}T^{\top})_{11})$, where $T = \left(\begin{array}{cc}
1 & -1\\ 1 & 1
\end{array}\right)$ and $(\cdot)_{11}$ denotes the $(1, 1)^{\rm th}$ element of a matrix. 
\label{lemma:bivariatenormal_plp}
\end{lemma}
\begin{proof}
The proof follows directly from Lemma~\ref{lemma:bivariateNormaldiff}.
\end{proof}

\section{Proof of Lemma~\ref{lemma:asymptotic_likelihood}}
\label{appendix:proof_laq}

\begin{proof}
Let the likelihood be $L_n(\theta)$. The definition of local asymptotic quadratics gives that, under regularity conditions, there exists random vectors $S_n$ and random matrices $K_n$ that are functions of data ${\bf y}$ such that when $t\leq b$ for some positive constant $b$,
\begin{align*}
    \log\frac{L_n(\theta+\frac{t}{\sqrt{n}})}{L_n(\theta)} - \left(t^\top S_n -\frac{1}{2} t^\top K_n t\right)\rightarrow 0
\end{align*}
in $P_{\theta, n}$ probability. We replace $\theta$ with $\theta_0$ and let $t = \sqrt{n}(\theta-\theta_0)$ in the expression above. Thus we have
\begin{align*}
    \log\frac{L_n(\theta)}{L_n(\theta_0)} - \left(\sqrt{n} (\theta-\theta_0) S_n -\frac{n}{2} (\theta-\theta_0)^\top K_n (\theta-\theta_0)\right)\rightarrow 0
\end{align*}
in $P_{\theta_0, n}$ probability, with $|\theta-\theta_0| \leq \frac{b}{\sqrt{n}}$. The conclusion follows by re-arranging the terms. 
\end{proof}

\section{Proof of Proposition~\ref{proposition:postmode_partial_delta}}

\begin{proof}
Then the log-posterior is
\begin{align*}
\log p(p_0, \eta, r | \boldsymbol{y}, \boldsymbol{n}) &= y_0 \log(p_0) + (n_0-y_0) \log(1-p_0) + y_1 \log(p_0+\eta)\\
&+ (n_1-y_1) \log(1-p_0-\eta) + {\rm Const.}\\
&- \frac{1}{2 (1-r^2)} \left[ \frac{(p_0-\alpha_0)^2}{\sigma_0^2} - 2r \frac{(p_0-\alpha_0)}{\sigma_0}\frac{(\eta-\eta_0)}{\sigma_1} + \frac{(\eta-\eta_0)^2}{\sigma_1^2} \right].
\end{align*}
For any $r$, by setting the derivative of the logarithm of posterior to zero, the posterior mode $(p_0^*,\eta^*)$ satisfies
\begin{align*}
&I_0 \left(\frac{y_0}{n_0}-p_0\right)+ I_1\left(\frac{y_1}{n_1} - p_0-\eta\right) -\frac{1}{1-r^2}\left[ \frac{p_0-\alpha_0}{\sigma_0^2} -r\frac{\eta-\eta_0}{\sigma_0\sigma_1} \right] = 0,\\ 
&I_1\left(\frac{y_1}{n_1} - p_0-\eta\right) - \frac{1}{1-r^2}\left[ \frac{\eta-\eta_0}{\sigma_1^2} -r\frac{p_0-\alpha_0}{\sigma_0\sigma_1} \right] = 0.
\end{align*}
Re-arranging the terms gives
\begin{equation*}
\left(\begin{array}{cc}
 I_0 + \frac{\frac{1}{\sigma_0^2} + \frac{r}{\sigma_0\sigma_1}}{1-r^2} & - \frac{\frac{r}{\sigma_0\sigma_1} + \frac{1}{\sigma_1^2}}{(1-r^2)} \\
I_1 - \frac{r}{\sigma_0\sigma_1(1-r^2)} & I_1 +\frac{1}{\sigma_1^2 (1-r^2)}\\
\end{array}\right) 
\left(\begin{array}{c}
p_0\\ \eta
\end{array}\right) =  
\left(\begin{array}{c}
\frac{y_0}{n_0} I_0  -\frac{1}{1-r^2} \left[-\left(\frac{1}{\sigma_0^2} + \frac{r}{\sigma_0\sigma_1}\right)\alpha_0 + \left(\frac{r}{\sigma_0\sigma_1}+\frac{1}{\sigma_1^2}\right) \eta_0\right]\\
\frac{y_1}{n_1} I_1 -\frac{1}{1-r^2}\left[-\frac{\eta_0}{\sigma_1^2}+ \frac{\alpha_0 r}{\sigma_0\sigma_1}\right]
\end{array}\right).
\end{equation*}
Therefore, the solution to the above equation is
\footnotesize
\begin{align*}
\eta &= \frac{\left[I_0 + \frac{\frac{1}{\sigma_0^2} + \frac{r}{\sigma_0\sigma_1}}{1-r^2}\right]\left[\frac{y_1}{n_1} I_1 -\frac{1}{1-r^2}\left[-\frac{\eta_0}{\sigma_1^2}+ \frac{\alpha_0 r}{\sigma_0\sigma_1}\right]\right] - \left[I_1 - \frac{r}{\sigma_0\sigma_1(1-r^2)}\right]\left[\frac{y_0}{n_0} I_0  -\frac{-\left(\frac{1}{\sigma_0^2} + \frac{r}{\sigma_0\sigma_1}\right)\alpha_0 + \left(\frac{r}{\sigma_0\sigma_1}+\frac{1}{\sigma_1^2}\right) \eta_0}{1-r^2}\right]}{\left[ I_0 + \frac{\frac{1}{\sigma_0^2} + \frac{r}{\sigma_0\sigma_1}}{1-r^2}\right]\left[I_1 +\frac{1}{\sigma_1^2 (1-r^2)}\right]+\frac{\frac{r}{\sigma_0\sigma_1} + \frac{1}{\sigma_1^2}}{(1-r^2)}\left[I_1 - \frac{r}{\sigma_0\sigma_1(1-r^2)}\right]}\\
&= W_L \left(\frac{y_1}{n_1}-\frac{y_0}{n_0}\right)  + W_{\eta} \eta_0 +  W_{d} \left(\frac{y_0}{n_0} - \alpha_0\right),
\end{align*}
\normalsize
where $W_L = \frac{(1-r^2)I_0 I_1 + I_1[\frac{1}{\sigma_0^2} + \frac{r}{\sigma_0\sigma_1}]}{(1-r^2)I_0 I_1+ \frac{1}{\sigma_0^2\sigma_1^2 } + \frac{I_0}{\sigma_1^2}+ I_1 \left[\frac{1}{\sigma_0^2}+ 2\frac{r}{\sigma_0\sigma_1}+\frac{1}{\sigma_1^2}\right]}$, $W_{\eta} = 1-W_L$, and 
\begin{equation*}
W_d = \frac{I_1\left(\frac{1}{\sigma_0^2} + \frac{r}{\sigma_0\sigma_1}\right)+ \frac{r}{\sigma_0\sigma_1} I_0 }{(1-r^2)I_0 I_1+ \frac{1}{\sigma_0^2\sigma_1^2} + \frac{I_0}{\sigma_1^2}+ I_1\left[\frac{1}{\sigma_0^2}+ 2\frac{r}{\sigma_0\sigma_1}+\frac{1}{\sigma_1^2}\right]}.
\end{equation*}
Since $\frac{y_0}{n_0}\geq \alpha_0$, DPP occurs if and only if
\begin{align*}
\left[\eta - \hat{\eta}\right] (\eta - \eta_0) &= -W_{\eta}W_L \left[ (\eta_0-\hat{\eta}) + \frac{W_d}{W_{\eta}} \left(\frac{y_0}{n_0} - \alpha_0\right)\right] \left[(\eta_0-\hat{\eta}) - \frac{W_d}{W_L} \left(\frac{y_0}{n_0} - \alpha_0\right)\right]\geq 0\\
\Leftrightarrow & \eta_0 - \hat{\eta} = \eta_0 -\left(\frac{y_1}{n_1} -\frac{y_0}{n_0}\right) \in \left[ -  \frac{W_d}{W_{\eta}} \left(\frac{y_0}{n_0} - \alpha_0\right),  \frac{W_d}{W_L} \left(\frac{y_0}{n_0} - \alpha_0\right) \right]. 
\end{align*}
\end{proof}

\section{Proof of Proposition~\ref{proposition:binomial_logit}}

\begin{proof}
Let $\tilde{\xi}_i = \frac{\xi_i - \mu_i}{\sigma_i}$ for $i = 0, 1$. The logarithm of the posterior distribution of $\boldsymbol{\theta}$ is then given by
\begin{align*}
\log p(\boldsymbol{\xi}, r | \boldsymbol{y}, \boldsymbol{n}) & = \log \phi(\boldsymbol{\xi}; \boldsymbol{\mu}, \Sigma) + \log  L\left( {\rm logit}^{-1} (\xi_0), {\rm logit}^{-1} (\xi_1) \right) + {\rm Const}\\
& =  - \frac{\tilde{\xi}_0^2 - 2r \tilde{\xi}_0\tilde{\xi}_1 + \tilde{\xi}_1^2}{2 (1-r^2)}  -\frac{1}{2} \log \left[\sigma_1^2 \sigma_0^2 (1 - r^2)\right]+ {\rm Const}\\
&+ y_0 \xi_0 - n_0\log(1+\exp(\xi_0)) + y_1 \xi_1 - n_1\log(1+\exp(\xi_1)). 
\end{align*}
Then $(\boldsymbol{\xi}^*, r^*)$, maximizer of the log posterior, satisfies the following equations
\begin{align*}
&y_0 - n_0 p_0 -\frac{\phi_0}{\sigma_0} = 0, y_1 - n_1 p_1 +\frac{\phi_1}{\sigma_1} = 0, \phi_0\phi_1 = \frac{r}{1-r^2}, 
\end{align*}
where $\phi_0= \frac{\tilde{\xi}_0-r \tilde{\xi}_1}{1-r^2}$ and $\phi_1 = \frac{r\tilde{\xi}_0- \tilde{\xi}_1}{1-r^2}$. Consequently, at the mode,
\begin{align*}
&\left(\frac{y_0}{n_0} - p_0\right)\left(\frac{y_1}{n_1} - p_1\right)  = -\frac{r}{n_0 n_1\sigma_0\sigma_1 (1-r^2)}.
\end{align*}
\end{proof}

\section{Binomial Model Large Sample Asymptotic DPP}
\label{appendix:dpp_binomial}

By the De Moivre–Laplace theorem, as $n$ grows large, for $k$ in a neighbourhood of $np$, we can approximate the Binomial likelihood ${\rm Binomial}(k; n, p)$ as 
\begin{align*}
    \left(\begin{array}{c}
         n \\
         k 
    \end{array}\right) p^k (1-p)^{n-k} \approx \frac{1}{\sqrt{2\pi n\hat{p}(1-\hat{p})}} \exp\left[- \frac{(k-np)^2}{2n\hat{p}(1-\hat{p})}\right],\quad \hat{p}= k/n.
\end{align*}
Therefore, in the model $y_j\sim {\rm Binomial}(n_j, p_j)$ with independent priors $p_j \sim {\rm Beta}(a_j+1, b_j+1)$, $j=0, 1$; we have, the posterior is proportional to
\begin{align*}
    \exp\left[- \sum_{j=0}^{1}\frac{ (y_j-n_jp_j)^2}{2 y_j (1-y_j/n_j)}\right] \prod_{j=0}^{1} \left[p_j^{a_j} (1-p_j)^{b_j}\right]. 
\end{align*}
Therefore, the posterior mode for $\eta=p_1-p_0$ is given by
\begin{align*}
    {\eta}^p = \frac{y_0+a_0}{n_0+a_0+b_0} - \frac{y_1+a_1}{n_1+a_1+b_1}.
\end{align*}
We note that the prior mode and the MLE for $\eta$ are given by 
\begin{align*}
    \eta_0=\frac{a_0}{a_0+b_0} - \frac{a_1}{a_1+b_1}, \quad \hat{\eta}=\frac{y_0}{n_0} - \frac{y_1}{n_1}.
\end{align*}
It is easy to derive that the DPP occurs if and only if
\begin{align*}
    \frac{\frac{a_0}{a_0+b_0} - \frac{y_0}{n_0}}{\frac{a_1}{a_1+b_1} - \frac{y_1}{n_1}} \in \text{ the range of } \left\{\frac{1-w_1}{1-w_0},\frac{w_1}{w_0}\right\}, \quad \text{where}\\
    w_0 = \frac{n_0}{n_0+a_0+b_0},\quad w_1 = \frac{n_1}{n_1+a_1+b_1}.
\end{align*}
This result corresponds, in very similar analytical form, to Example~\ref{proposition:2dimconstrast_heter_uncor_cov} in the exponential quadratic case, despite that we are using Beta priors here instead of Gaussian priors. 
\end{appendix}
\end{document}